\documentclass[11pt,reqno]{amsart}

\usepackage[a4paper,margin=1in]{geometry}
\usepackage{amsmath,amssymb,amsfonts,amsthm,mathtools,mathrsfs}
\usepackage{microtype}
\usepackage{enumitem}
\usepackage{booktabs}
\usepackage[hidelinks]{hyperref}
\usepackage[nameinlink,capitalize,noabbrev]{cleveref}
\allowdisplaybreaks
\numberwithin{equation}{section}

\title[A compensated Piola principle]
{A compensated Piola principle for critical nondiffusive parabolic systems}
\author{Maotuo Guo}
\address{Department of Mathematics, Harbin University of Science and Technology, Harbin, China}
\email{guomaotuo@hrbust.edu.cn}
\date{}

\newcommand{\R}{\mathbb R}
\newcommand{\B}{\dot B}

\newcommand{\diver}{\operatorname{div}}
\newcommand{\Id}{\operatorname{Id}}
\newcommand{\T}{\dot T}
\newcommand{\Rem}{\dot R}
\newcommand{\D}{\mathrm D}
\newcommand{\dd}{\,\mathrm d}
\newcommand{\PP}{\mathbb P}
\newcommand{\QQ}{\mathbb Q}
\newcommand{\K}{\mathcal K}
\newcommand{\cF}{\mathcal F}
\newcommand{\cZ}{\mathcal Z}

\newcommand{\cM}{\mathcal M}

\newcommand{\norm}[2]{\left\lVert #1\right\rVert_{#2}}

\newtheorem{theorem}{Theorem}[section]
\newtheorem{proposition}[theorem]{Proposition}
\newtheorem{lemma}[theorem]{Lemma}
\newtheorem{corollary}[theorem]{Corollary}

\theoremstyle{definition}
\newtheorem{definition}[theorem]{Definition}
\theoremstyle{remark}
\newtheorem{remark}[theorem]{Remark}
\crefname{theorem}{Theorem}{Theorems}
\crefname{proposition}{Proposition}{Propositions}
\crefname{lemma}{Lemma}{Lemmas}
\crefname{corollary}{Corollary}{Corollaries}
\crefname{claim}{Claim}{Claims}
\crefname{definition}{Definition}{Definitions}
\crefname{remark}{Remark}{Remarks}

\begin{document}

\begin{abstract}
We introduce a compensated Piola graph method for Lagrangian stability estimates in critical homogeneous Besov spaces below the usual product threshold.  For $d\ge2$, $2d\le p<\infty$, and $s=d/p$, we establish local Hadamard well-posedness and a continuation criterion in the scaling-critical Besov phase space $\dot B^{s-1}_{p,1}(\mathbb R^d)^d\times \dot B^s_{p,1}(\mathbb R^d)$, for a class of incompressible parabolic systems coupled to nondiffusive internal variables.  The main obstruction is that the formal Piola product between the inverse deformation gradient and the Lagrangian velocity is only borderline at $p=2d$ and is not continuous in general for $p>2d$.  We replace this product by a closed solenoidal Piola graph, using a compensated divergence structure that survives in the high-integrability range.  The principle applies to viscous non-resistive MHD, Hookean incompressible viscoelasticity, and nondiffusive Oldroyd--B systems with affine objective terms.  In particular, it closes the previously untreated high-$p$ Hadamard well-posedness range for critical non-resistive MHD and for the nondiffusive Oldroyd--B systems considered here; combined with the known low-$p$ MHD theory, it gives the finite-$p$, $q=1$ critical Besov picture for non-resistive MHD.
\end{abstract}
\maketitle

\section{Introduction}

\subsection{The threshold and the geometric obstruction}
A recurring critical-space model in incompressible continuum mechanics consists of a viscous velocity $u$, a pressure $P$, and a nondiffusive internal variable $r$ transported and deformed by the flow.  Scaling assigns one fewer derivative to $u$ than to $r$:
\[
 u_0\in\dot B^{d/p-1}_{p,1},\qquad r_0\in\dot B^{d/p}_{p,1}.
\]
Put $s=d/p$.  The standard Eulerian difference estimates use products at regularity $s-1$, while the Lagrangian formulation formally introduces the Piola pullback $Av$, with $X$ the volume-preserving flow, $A=(DX)^{-1}$, and $v=u\circ X$.  The expected generic multiplication is
\begin{equation}\label{eq:intro-formal-product}
 \dot B^s_{p,1}\cdot\dot B^{s-1}_{p,1}
 \longrightarrow\dot B^{s-1}_{p,1}.
\end{equation}
At $p=2d$, equivalently $s=1/2$, this is a borderline product for the usual Besov calculus.  In the strict high-$p$ range $p>2d$, equivalently $s<1/2$, the map in \eqref{eq:intro-formal-product} is not continuous; Proposition~\ref{prop:generic-failure} gives an explicit high--high counterexample.  Thus the obstruction is not a loss caused by a particular constitutive law.  It is a geometric obstruction in the rough Piola transform itself.

For viscous non-resistive MHD, the critical Besov picture is especially sharp. Earlier Sobolev and nearly critical treatments of the same zero-resistivity difficulty include \cite{FeffermanEtAlJFA,FeffermanEtAlARMA}.  In the $q=1$ critical Besov scale, the $p=2$ theory was initiated in \cite{CheminEtAl}, and homogeneous Besov local existence and uniqueness were developed further in \cite{LiTanYin}.  The lifespan and Hadamard theory of \cite{YeLuoYin} gives local existence for all finite $p$ and continuous dependence in the range $p\le2d$.  For $q>1$, known norm-inflation results indicate the sharpness of the $\ell^1$ Besov summation in the mathematical critical Besov scale; see \cite{ChenNieYe}.  Thus, within the finite-$p$, $q=1$ scale, the remaining Hadamard gap is the strict high-integrability range $p>2d$.  The new MHD content of this paper is precisely the high-$p$ stability and uniqueness mechanism, with the endpoint $p=2d$ treated by the same uniform argument rather than by a separate borderline product estimate.  The threshold can be summarized as follows:
\[
\begin{array}{c|c|c}
 \text{range} & \text{generic product }\dot B^s_{p,1}\cdot\dot B^{s-1}_{p,1} & \text{stability mechanism}\\
\hline
 p<2d & \text{available} & \text{standard critical calculus}\\
 p=2d & \text{borderline} & \text{compensated Piola graph}\\
 p>2d & \text{false in general} & \text{closed solenoidal Piola graph}
\end{array}
\]
For Hookean incompressible viscoelasticity, critical and near-critical theories close to the identity, often coupled to geometric compatibility conditions and smallness assumptions for global results, include \cite{CheminMasmoudi,ChenMiao,Qian,LinLiuZhang,LeiLiuZhou,XieFu,ZhangFangLp}; for a recent critical Besov result in Hookean elastodynamics, see also \cite{ZhangZhouHookean}.  Our result is a local large-data theorem in the high-$p$ critical topology.  The analytic well-posedness statement does not require the physical compatibility constraints, although those constraints are propagated when imposed initially.

For Oldroyd-type models, the constitutive background goes back to \cite{Oldroyd1950,BirdArmstrongHassager,Renardy}; classical and modern PDE results include \cite{GuillopeSaut,LionsMasmoudi,ConstantinKliegl,ElgindiLiu,ElgindiRousset,ChenHao}.  For the principal source-free nondiffusive Oldroyd--B model, the local critical Besov theory quoted from \cite{DeAnnaPaicu} assumes $1\le p<2d$ in dimensions $d\ge3$, while the analogous $q>1$ ill-posedness in the principal Oldroyd--B critical Besov scale is proved in \cite{LiYuZhu}.  We cover $p\ge2d$ for the affine-in-gradient objective laws stated below, including damping, source, and both corotational and non-corotational terms.

Lagrangian--Eulerian methods for rough incompressible coefficients are represented by \cite{CheminBook,ConstantinLE,DanchinMucha}.  The point here is to make the Piola pullback meaningful below the generic product-law threshold.  Although Piola terminology is classical for divergence-preserving transformations, for instance in mixed finite elements \cite{GiraultRaviart,BoffiBrezziFortin}, the present use is an analytic critical-Besov graph construction rather than a finite-element mapping argument.

The central observation can be summarized in one formula.  The rough product
\[
 Av=(I+(A-I))v
\]
is not the object to be estimated.  Instead one defines the Piola pullback by the graph
\begin{equation}\label{eq:intro-main-formula}
 v=Fz=z+(z\cdot\nabla)Y,\qquad \operatorname{div}z=0.
\end{equation}
The high--high remainder in $(z\cdot\nabla)Y$ is a divergence, and hence satisfies a compensated estimate that is unavailable for arbitrary factors.  The paper turns \eqref{eq:intro-main-formula} into a closed graph on the solenoidal critical Besov space, proves that the graph can be differentiated in time, and uses it to control the nonzero divergence arising when two Lagrangian solutions are subtracted.  This closed graph principle is the part of the argument meant to be reusable beyond the three model systems treated here.

\subsection{Admissible transport--stress laws}
Let $\mathcal V$ be a finite-dimensional real inner-product space, let $\mathcal W\subset\mathcal V$ be a fixed linear subspace, and let $\bar a\in\mathcal V$ be a constant equilibrium.  Let
\[
 \mathbb S:\mathcal V\to\mathbb R^{d\times d},\qquad
 \mathbb G:\mathcal V\times\mathbb R^{d\times d}\to\mathcal V
\]
be constitutive maps.  We consider
\begin{equation}\label{eq:general-system}
\left\{
\begin{aligned}
 &\partial_tu+u\cdot\nabla u-\nu\Delta u+\nabla P
   =\operatorname{div}\mathscr S(r),\\
 &\partial_tr+u\cdot\nabla r=\mathscr G(r,\nabla u),\\
 &\operatorname{div}u=0,\\
 &(u,r)|_{t=0}=(u_0,r_0),
\end{aligned}
\right.
\qquad (t,x)\in[0,T]\times\mathbb R^d,
\end{equation}
where $\nu>0$ and
\[
 \mathscr S(r):=\mathbb S(\bar a+r)-\mathbb S(\bar a),
 \qquad
 \mathscr G(r,M):=\mathbb G(\bar a+r,M).
\]
Subtracting the constant equilibrium stress is immaterial after taking divergence and is essential for a clean homogeneous-space formulation.

\begin{definition}[Admissible law at an invariant affine space]\label{def:admissible}
The quadruple $(\bar a,\mathcal W,\mathbb S,\mathbb G)$ is called admissible if:
\begin{enumerate}[label=\textup{(A\arabic*)}]
 \item $\mathbb S$ is $C^\infty$, with all derivatives bounded on bounded subsets of $\mathcal V$;
 \item there are smooth maps $g:\mathcal V\to\mathcal V$ and
 $\mathbb L:\mathcal V\to\mathcal L(\mathbb R^{d\times d},\mathcal V)$, with derivatives bounded on bounded sets, such that
 \[
  \mathbb G(a,M)=g(a)+\mathbb L(a)[M],\qquad g(\bar a)=0;
 \]
 \item the affine constraint $\bar a+\mathcal W$ is invariant: for every $w\in\mathcal W$ and every matrix $M$,
 \[
  g(\bar a+w)\in\mathcal W,
  \qquad \mathbb L(\bar a+w)[M]\in\mathcal W.
 \]
\end{enumerate}
When no finite-dimensional constraint is present, one takes $\mathcal W=\mathcal V$.
\end{definition}

The restriction in (A2) is part of the theorem: the internal law is smooth in the internal variable and affine in the velocity gradient.  Nonlinear dependences on $M$, such as quadratic gradient laws or shear-dependent viscosities, are not included in the abstract result.  Differential or nonlinear physical constraints, for example solenoidality, determinant constraints, commutation of vector fields, or positivity of a conformation tensor, are also not encoded in the finite-dimensional subspace $\mathcal W$; when relevant they are propagated separately in Section~\ref{sec:applications}.

Set $s=d/p$ and
\begin{align*}
 \mathcal E_{p,\mathcal W}(T):={}&
 \Bigl(C([0,T];\dot B^{s-1}_{p,1})
 \cap L^1(0,T;\dot B^{s+1}_{p,1})\Bigr)^d\\
 &\times C([0,T];\dot B^s_{p,1};\mathcal W).
\end{align*}

Throughout the paper, local Hadamard well-posedness means existence and uniqueness in $\mathcal E_{p,\mathcal W}(T)$ together with continuity of the data-to-solution map in the topology displayed above, on every common compact subinterval of the lifespan.

\begin{theorem}[Compensated Piola well-posedness principle]\label{thm:abstract-main}
Let $d\ge2$, $2d\le p<\infty$, $s=d/p$, and let
$(\bar a,\mathcal W,\mathbb S,\mathbb G)$ be admissible.  For every
\[
 u_0\in\dot B^{s-1}_{p,1}(\mathbb R^d)^d,\qquad
 r_0\in\dot B^s_{p,1}(\mathbb R^d;\mathcal W),\qquad
 \operatorname{div}u_0=0,
\]
there is $T>0$ for which \eqref{eq:general-system} has a unique solution
$(u,r)\in\mathcal E_{p,\mathcal W}(T)$.  The pressure is determined modulo constants and may be chosen in $L^1(0,T;\dot B^s_{p,1})$.  On every compact subinterval of the maximal lifespan, the data-to-solution map is continuous in the stated Eulerian phase space.  On a sufficiently short common interval, the corresponding Lagrangian solution map is locally Lipschitz.

If $T^*<\infty$ is the maximal lifespan and
\[
 \sup_{t<T^*}\Bigl(
 \|u(t)\|_{\dot B^{s-1}_{p,1}}+
 \|r(t)\|_{\dot B^s_{p,1}}\Bigr)
 +\int_0^{T^*}\|u(t)\|_{\dot B^{s+1}_{p,1}}\,dt<\infty,
\]
then the solution extends beyond $T^*$.
\end{theorem}

\subsection{Three model consequences}
The three applications have genuinely different internal variables: a solenoidal frozen-in vector, a matrix deformation gradient with an identity background, and an objective stress tensor satisfying a Lagrangian matrix ODE.

\begin{theorem}[High-$p$ non-resistive MHD]\label{cor:mhd-main}
Let $d\ge2$, $2d\le p<\infty$, and $s=d/p$.  The viscous non-resistive incompressible MHD system
\begin{equation}\label{eq:MHD-intro}
\left\{
\begin{aligned}
 &\partial_tu+u\cdot\nabla u-\nu\Delta u+\nabla P
   =\operatorname{div}(b\otimes b),\\
 &\partial_tb+u\cdot\nabla b=(\nabla u)b,\\
 &\operatorname{div}u=\operatorname{div}b=0
\end{aligned}
\right.
\end{equation}
is locally Hadamard well posed for divergence-free data
$(u_0,b_0)\in\dot B^{s-1}_{p,1}\times\dot B^s_{p,1}$.
\end{theorem}

\begin{corollary}[Finite MHD integrability exponents]\label{cor:mhd-all-p}
Combining Theorem~\ref{cor:mhd-main} with the $q=1$ critical Besov theory of \cite{YeLuoYin} in the range $1<p\le2d$, system \eqref{eq:MHD-intro} is locally Hadamard well posed for divergence-free data in
\[
 \dot B^{d/p-1}_{p,1}\times\dot B^{d/p}_{p,1}
\]
for every $1<p<\infty$.
\end{corollary}

\begin{remark}\label{rem:low-p-not-reproved}
The low-$p$ part of Corollary~\ref{cor:mhd-all-p} is quoted from \cite{YeLuoYin}: that work gives local existence for finite $p$ and continuous dependence in the range $p\le2d$.  It is not reproved here.  The contribution of the present paper is the stability mechanism for $2d\le p<\infty$, in particular the strict range $p>2d$ where the generic product in \eqref{eq:intro-formal-product} fails.  If one invokes a low-$p$ theorem including the endpoint $p=1$, that endpoint can be appended; no argument in the present paper uses $p=1$.
\end{remark}

\begin{theorem}[Hookean incompressible viscoelasticity]\label{cor:hookean-main}
Let $U=I+H$ denote the deformation gradient and consider
\begin{equation}\label{eq:Hookean-intro}
\left\{
\begin{aligned}
 &\partial_tu+u\cdot\nabla u-\nu\Delta u+\nabla P
   =\operatorname{div}(UU^\top-I),\\
 &\partial_tU+u\cdot\nabla U=(\nabla u)U,\\
 &\operatorname{div}u=0.
\end{aligned}
\right.
\end{equation}
For $2d\le p<\infty$, with $s=d/p$, the system is locally Hadamard well posed for
\[
 u_0\in\dot B^{s-1}_{p,1},\qquad H_0=U_0-I\in\dot B^s_{p,1}.
\]
The analytic conclusion is first proved on the full affine Besov space $I+\dot B^s_{p,1}$ and does not require the physical constraints.  If initially
\[
 \operatorname{div}U_0^\top=0,\qquad \det U_0=1,
 \qquad [U_0^{(\alpha)},U_0^{(\beta)}]=0
 \quad(1\le\alpha,\beta\le d),
\]
where $U^{(\alpha)}$ are the columns, $\det U_0=1$ is understood almost everywhere through the affine Besov algebra $I+\dot B^s_{p,1}$, and the Lie brackets are understood in the affine distributional sense made precise in Proposition~\ref{prop:hookean-constraints}, then all three constraints are propagated.
\end{theorem}

\begin{theorem}[Nondiffusive Oldroyd--B]\label{cor:oldroyd-main}
Let $a,\mu_1,\mu_2\ge0$, $b\in[-1,1]$, and define
\[
 D(M)=\tfrac12(M+M^\top),\qquad
 \Omega(M)=\tfrac12(M-M^\top),
\]
\[
 Q_b(\tau,M)=\tau\Omega(M)-\Omega(M)\tau
              +b\bigl(D(M)\tau+\tau D(M)\bigr).
\]
For $2d\le p<\infty$, with $s=d/p$, the system
\begin{equation}\label{eq:Oldroyd-intro}
\left\{
\begin{aligned}
 &\partial_tu+u\cdot\nabla u-\nu\Delta u+\nabla P
   =\mu_1\operatorname{div}\tau,\\
 &\partial_t\tau+u\cdot\nabla\tau+a\tau+Q_b(\tau,\nabla u)
   =\mu_2D(\nabla u),\\
 &\operatorname{div}u=0
\end{aligned}
\right.
\end{equation}
is locally Hadamard well posed for
$(u_0,\tau_0)\in\dot B^{s-1}_{p,1}\times\dot B^s_{p,1}$, with $\tau_0$ symmetric.  Symmetry of $\tau$ is propagated.  Positivity or conformation-tensor cone constraints are not part of this theorem.  The conclusion includes the corotational case $b=0$, the non-corotational case $b\ne0$, the undamped case $a=0$, and the source-free case $\mu_2=0$; in particular, it covers the infinite-Weissenberg/source-free regime $a=\mu_2=0$.
\end{theorem}

\subsection{Scope of the contribution and proof architecture}
The common result is not obtained by repeating three model-specific proofs.  The constitutive variable always remains at the algebra regularity $\dot B^s_{p,1}$ after composition with the flow; the genuinely new issue is the velocity geometry at regularity $s-1$.  The argument has four steps.
\begin{enumerate}[label=\textup{(\roman*)}]
 \item \emph{Compensated product.}  For a solenoidal field $z$,
 \begin{equation}\label{eq:intro-null}
  \|(z\cdot\nabla)Y\|_{\dot B^{s-1}_{p,1}}
  \lesssim
  \|z\|_{\dot B^{s-1}_{p,1}}
  \|Y\|_{\dot B^{s+1}_{p,1}}.
 \end{equation}
 The divergence-free constraint is used precisely in the high--high remainder.  A separate counterexample shows that, when $p>2d$, neither the corresponding generic product nor the same differential expression without the solenoidal condition is bounded.
 \item \emph{Closed Piola graph.}  For $X=\operatorname{Id}+Y$, $F=DX$, and $v=u\circ X$, the negative-index product $F^{-1}v$ is never formed.  Instead, the Piola pullback is defined as the unique $z$ in the solenoidal subspace satisfying
 \begin{equation}\label{eq:intro-graph}
  v=L_Yz:=z+(z\cdot\nabla)Y.
 \end{equation}
 Smallness of $Y$ makes $L_Y$ coercive on that subspace, and a dyadic Fatou argument proves that this graph is closed under critical approximation.
 \item \emph{Time differentiation and divergence control.}  A vector-valued BV product rule, applied with the correct solenoidal domain, excludes a singular time derivative of $z$.  This justifies differentiating \eqref{eq:intro-graph}.  For two flows, the resulting identities express both $\operatorname{div}(v_1-v_2)$ and its time derivative through compensated products, which are perturbative on short intervals.
 \item \emph{Abstract closure.}  The difference equation is reduced to a constant-coefficient Stokes system with nonzero divergence.  The internal variable solves a Banach-space ODE in the algebra $\dot B^s_{p,1}$, so the constitutive law enters only through tame algebra estimates.  This separation yields the abstract theorem and the three applications.
\end{enumerate}

The comparison with the closest critical results is consequently precise.  In non-resistive MHD, the paper supplies continuous dependence in the previously untreated strict range $p>2d$ and gives an independent endpoint argument at $p=2d$.  In the Hookean system, it gives local high-$p$ Hadamard well-posedness without imposing the nonlinear physical constraints as hypotheses for the analytic construction; a separate push-forward argument proves propagation of those constraints.  In nondiffusive Oldroyd--B, it supplies the high-$p$ local theory for the affine-in-gradient objective family displayed in \eqref{eq:Oldroyd-intro}.  These conclusions use the same Piola graph but genuinely different internal-variable dynamics: a frozen-in solenoidal vector, an affine deformation gradient, and an objective stress ODE.

The paper is organized as follows.  Section~\ref{sec:functional} records the functional setting, the composition, Fatou, and endpoint transport estimates, and the local construction.  Section~\ref{sec:null} proves the compensated product and its strict high-$p$ contrast with generic multiplication.  Section~\ref{sec:piola} develops the closed Piola graph and its BV differentiation.  Sections~\ref{sec:lagrangian}--\ref{sec:stability} establish the abstract Lagrangian stability principle.  Section~\ref{sec:applications} verifies the three models and their invariant constraints.  The appendices provide the local construction and the vector-valued BV input.

\section{Functional setting and local construction}\label{sec:functional}

Throughout the compensated argument we fix
\begin{equation}\label{eq:range}
 2d\le p<\infty,\qquad s:=\frac dp\in\Bigl(0,\frac12\Bigr],
 \qquad \sigma:=s-1\in\Bigl(-1,-\frac12\Bigr].
\end{equation}
All homogeneous Besov spaces are realized in the usual subspace $\mathcal S'_h$ of tempered distributions.  Vector-, matrix-, and $\mathcal W$-valued norms are understood componentwise.  Pressures are taken modulo constants.

Constant matrices do not belong to homogeneous Besov spaces.  Accordingly, expressions such as $A=I+\alpha$ or $U=I+H$ are interpreted in the affine algebra
\[
 I+\dot B^s_{p,1}:=\{I+f:f\in\dot B^s_{p,1}\},
 \qquad \|I+f\|_{I+\dot B^s}:=1+\|f\|_{\dot B^s}.
\]
Every product involving an affine factor is expanded, for example
$SA^\top=S+S(A-I)^\top$.  This convention is used throughout and prevents constants from being assigned a homogeneous Besov norm.

We use the shorthand $C_TX=C([0,T];X)$ and $L^q_TX=L^q(0,T;X)$.  All time-dependent norms in the proof are ordinary Bochner norms.

\begin{lemma}[Algebra, derivatives, interpolation, and smooth maps]\label{lem:standard-besov}
Let $1<p<\infty$ and $s=d/p>0$.
\begin{enumerate}[label=\textup{(\roman*)}]
 \item $\B^s_{p,1}\hookrightarrow L^\infty$ and
 \[
  \norm{fg}{\B^s_{p,1}}
  \lesssim \norm{f}{\B^s_{p,1}}\norm{g}{\B^s_{p,1}}.
 \]
 \item For homogeneous representatives,
 \[
  \norm{\nabla f}{\B^r_{p,1}}\simeq\norm{f}{\B^{r+1}_{p,1}},
  \qquad
  \norm{\nabla^2 f}{\B^r_{p,1}}\simeq\norm{f}{\B^{r+2}_{p,1}}.
 \]
 \item If $f\in L^\infty_T\dot B^{r-1}_{p,1}\cap L^1_T\dot B^{r+1}_{p,1}$, then
 \[
  \|f\|_{L^2_T\dot B^r_{p,1}}^2
  \lesssim
  \|f\|_{L^\infty_T\dot B^{r-1}_{p,1}}
  \|f\|_{L^1_T\dot B^{r+1}_{p,1}}.
 \]
 \item Let $\Phi$ be smooth on an open neighbourhood of the ranges of $f_1$ and $f_2$, with $\Phi(0)=0$.  For each $R>0$ there is $C_R$ such that, whenever
 $\norm{f_i}{\B^s_{p,1}}\le R$,
 \[
  \norm{\Phi(f_1)-\Phi(f_2)}{\B^s_{p,1}}
  \le C_R\norm{f_1-f_2}{\B^s_{p,1}}.
 \]
 The corresponding tame estimates hold for smooth finite-dimensional multilinear maps.
\end{enumerate}
\end{lemma}

\begin{proof}
These are the standard homogeneous Besov algebra, Fourier multiplier, interpolation, and smooth-composition estimates; see \cite{BCD,Triebel,RunstSickel,Meyer}, and \cite{BerghLofstrom} for the interpolation background.  The smooth-map estimate follows by writing the difference through the fundamental theorem of calculus and using the algebra property.
\end{proof}

\begin{lemma}[Bochner--Besov Fatou property]\label{lem:bochner-besov-fatou}
Let $1<p<\infty$, $\gamma\in\mathbb R$, and $1\le q\le\infty$.  Let $(f_n)$ be bounded in
$L^q(0,T;\dot B^\gamma_{p,1})$ and suppose that
$f_n\to f$ in $\mathcal D'((0,T)\times\mathbb R^d)$.  In the case $q=1$, assume in addition that the scalar functions
\[
 t\longmapsto \|f_n(t)\|_{\dot B^\gamma_{p,1}}
\]
are uniformly integrable on $(0,T)$.  Then
$f\in L^q(0,T;\dot B^\gamma_{p,1})$ and
\begin{equation}\label{eq:bochner-fatou}
 \|f\|_{L^q_T\dot B^\gamma_{p,1}}
 \le \liminf_{n\to\infty}
 \|f_n\|_{L^q_T\dot B^\gamma_{p,1}} .
\end{equation}
The same conclusion holds componentwise for vector- or finite-dimensional-valued functions, and simultaneously for intersections of finitely many such Bochner--Besov spaces.
\end{lemma}

\begin{proof}
For a finite set $J\subset\mathbb Z$, set
\[
 N_J(f)(t)=\sum_{j\in J}2^{j\gamma}
          \|\dot\Delta_jf(t)\|_{L^p}.
\]
The blocks $\dot\Delta_jf_n$ converge to $\dot\Delta_jf$ in distributions.  If $1<q<\infty$, reflexivity gives weak convergence, after extraction, in the finite product of Bochner spaces
$L^q(0,T;(L^p)^J)$.  If $q=\infty$, Banach--Alaoglu gives weak-star convergence in
$L^\infty(0,T;(L^p)^J)$.  If $q=1$, the additional uniform integrability assumption and the reflexivity of $L^p$ give weak compactness in
$L^1(0,T;(L^p)^J)$ for each finite $J$; the distributional limit identifies the weak limit.  In all three cases, lower semicontinuity of the convex functional
\[
 (g_j)_{j\in J}\longmapsto
 \left\|\sum_{j\in J}2^{j\gamma}\|g_j(t)\|_{L^p}\right\|_{L^q(0,T)}
\]
yields
\[
 \|N_J(f)\|_{L^q(0,T)}
 \le \liminf_{n\to\infty}\|N_J(f_n)\|_{L^q(0,T)} .
\]
Since $N_J(f)$ increases to
$\sum_j2^{j\gamma}\|\dot\Delta_jf\|_{L^p}$ as $J\uparrow\mathbb Z$, monotone convergence for $q<\infty$, and the corresponding essential-supremum monotone convergence for $q=\infty$, yield \eqref{eq:bochner-fatou}.  This is the standard Fatou property for homogeneous Besov spaces realized in $\mathcal S'_h$; see also \cite[Chapter~2]{BCD}.
\end{proof}

\begin{remark}\label{rem:fatou-q1}
The extra hypothesis in the endpoint $q=1$ is not cosmetic.  A bounded sequence in $L^1(0,T;X)$ may concentrate in time and converge in distributions to a vector-valued measure rather than to a Bochner function.  Thus a pure distributional-convergence hypothesis is sufficient for $q>1$ and $q=\infty$, but not for $q=1$.
\end{remark}

\begin{lemma}[Composition by time-dependent volume-preserving flows]\label{lem:composition}
Let $1<p<\infty$ and $-1<r<1$.  If $\Phi$ is a volume-preserving bi-Lipschitz homeomorphism of $\mathbb R^d$, then
\begin{equation}\label{eq:composition-bound}
 \|f\circ\Phi\|_{\dot B^r_{p,1}}
 \le C_r\bigl(\operatorname{Lip}\Phi,
                \operatorname{Lip}\Phi^{-1}\bigr)
       \|f\|_{\dot B^r_{p,1}}.
\end{equation}
Suppose, moreover, that all $\Phi_n$ are volume-preserving bi-Lipschitz homeomorphisms, that $\Phi_n,\Phi_n^{-1}$ have uniform Lipschitz bounds, and that
$\Phi_n\to\Phi$ uniformly.  Then, for every fixed
$f\in\dot B^r_{p,1}$,
\begin{equation}\label{eq:composition-strong}
 f\circ\Phi_n\longrightarrow f\circ\Phi
 \quad\hbox{in }\dot B^r_{p,1}.
\end{equation}
If the maps depend on time, converge uniformly on
$[0,T]\times\mathbb R^d$, and have uniform bi-Lipschitz bounds, then:
\begin{enumerate}[label=\textup{(\roman*)}]
 \item for $1\le q<\infty$ and $f\in L^q(0,T;\dot B^r_{p,1})$, convergence holds in the same Bochner space;
 \item for $f\in C([0,T];\dot B^r_{p,1})$, convergence is uniform in time.
\end{enumerate}
The bounds are uniform for the family of maps.
\end{lemma}

\begin{proof}
This is the standard volume-preserving composition theorem in the range $-1<r<1$; see the homogeneous composition results in \cite{BCD,CheminBook}.  We work throughout in the homogeneous realization $\mathcal S'_h$; the polynomial ambiguity and the low-frequency tail are handled by finite Littlewood--Paley truncations and the $\ell^1$ Besov summability.  We use only the measure-preserving property and the two Lipschitz constants, so the statement applies to the critical flows constructed below.  The proof is recalled to record the uniform-in-time strong convergence used later.  Write $T_\Phi f=f\circ\Phi$.  If $0<r<1$, the difference-quotient characterization of $\dot B^r_{p,1}$ and the identities
\[
 \|f\circ\Phi\|_{L^p}=\|f\|_{L^p},
 \qquad
 |\Phi(x+h)-\Phi(x)|\le \operatorname{Lip}\Phi\,|h|
\]
give \eqref{eq:composition-bound}; applying the same argument to $\Phi^{-1}$ controls the constants uniformly in terms of the two Lipschitz constants.  If $r=-\rho\in(-1,0)$, then for Schwartz functions, and hence by density in the homogeneous realization $\mathcal S'_h$,
\[
 \langle T_\Phi f,\varphi\rangle
 =\langle f,\varphi\circ\Phi^{-1}\rangle.
\]
The positive-index composition bound on $\dot B^\rho_{p',\infty}$ and the duality between $\dot B^{-\rho}_{p,1}$ and $\dot B^\rho_{p',\infty}$ give the negative-index estimate.  The case $r=0$ follows by interpolation between two indices $-\varepsilon$ and $\varepsilon$.

We next prove strong convergence.  First let $f$ be a finite Littlewood--Paley sum.  Then $f\in L^p$ and may be approximated in $L^p$ by functions in $C_c^\infty$.  For such a compactly supported smooth function, uniform convergence of $\Phi_n$ to $\Phi$, together with the uniform bi-Lipschitz bounds, gives pointwise convergence and a common bounded support for the composed functions; hence
\[
 \|f\circ\Phi_n-f\circ\Phi\|_{L^p}\longrightarrow0.
\]
The same conclusion for a finite Littlewood--Paley sum follows from the $L^p$ isometry of volume-preserving composition and density.

Choose $-1<a<r<b<1$.  The functions
$h_n=f\circ\Phi_n-f\circ\Phi$ are uniformly bounded in both
$\dot B^a_{p,1}$ and $\dot B^b_{p,1}$.  For an integer $J>0$, split the Besov norm into low, middle, and high frequencies.  The uniform endpoint bounds give
\[
 \sum_{j<-J}2^{jr}\|\dot\Delta_jh_n\|_{L^p}
 \lesssim 2^{-J(r-a)}\|h_n\|_{\dot B^a_{p,1}},
\]
\[
 \sum_{j>J}2^{jr}\|\dot\Delta_jh_n\|_{L^p}
 \lesssim 2^{-J(b-r)}\|h_n\|_{\dot B^b_{p,1}},
\]
whereas the finitely many middle blocks tend to zero because
$h_n\to0$ in $L^p$.  Letting first $n\to\infty$ and then $J\to\infty$ proves \eqref{eq:composition-strong} for finite sums.  For a general $f\in\dot B^r_{p,1}$, approximate $f$ by its finite Littlewood--Paley sums and use the uniform operator bound \eqref{eq:composition-bound}.  The use of the summability index $1$ is important here, since those finite sums converge in the full homogeneous Besov norm.

If the maps depend on time and converge uniformly on
$[0,T]\times\mathbb R^d$, the preceding proof is uniform in $t$: for compactly supported smooth $f$, uniform continuity gives uniform $L^p$ convergence, and all frequency-splitting constants depend only on the common bi-Lipschitz bound.  Therefore, for each fixed $f\in\dot B^r_{p,1}$,
\begin{equation}\label{eq:composition-uniform-fixed}
 \sup_{0\le t\le T}
 \|f\circ\Phi_n(t)-f\circ\Phi(t)\|_{\dot B^r_{p,1}}
 \longrightarrow0.
\end{equation}
For $f\in L^q_T\dot B^r_{p,1}$, pointwise convergence and
\[
 \|f(t)\circ\Phi_n(t)-f(t)\circ\Phi(t)\|_{\dot B^r_{p,1}}
 \le C\|f(t)\|_{\dot B^r_{p,1}}
\]
allow the use of dominated convergence.  Finally, if
$f\in C_T\dot B^r_{p,1}$, its range is compact.  Cover that compact set by a finite $\varepsilon$-net, apply \eqref{eq:composition-uniform-fixed} to the finitely many centers, and use the uniform operator bound for the approximation error.  This proves uniform convergence in time.
\end{proof}

\begin{lemma}[Tame bounds for admissible laws]\label{lem:constitutive-tame}
Let $(\bar a,\mathcal W,\mathbb S,\mathbb G)$ be admissible.  For every $R>0$ there is $C_R$ such that, if
$\norm{r_i}{\B^s_{p,1}}\le R$ and $M_i\in\dot B^s_{p,1}$, then
\begin{align}
 \norm{\mathscr S(r_1)-\mathscr S(r_2)}{\B^s_{p,1}}
 &\le C_R\norm{r_1-r_2}{\B^s_{p,1}},\label{eq:stress-tame}\\
 \norm{\mathscr G(r_1,M_1)-\mathscr G(r_2,M_2)}{\B^s_{p,1}}
 &\le C_R\bigl(1+\norm{M_1}{\B^s_{p,1}}
                    +\norm{M_2}{\B^s_{p,1}}\bigr)
       \norm{r_1-r_2}{\B^s_{p,1}}\notag\\
 &\quad+C_R\norm{M_1-M_2}{\B^s_{p,1}}.
 \label{eq:G-tame}
\end{align}
Moreover, if $\norm{r}{\B^s_{p,1}}\le R$ and $M\in\dot B^s_{p,1}$, then
\begin{equation}\label{eq:G-single}
 \norm{\mathscr G(r,M)}{\B^s_{p,1}}
 \le C_R\bigl(\norm{r}{\B^s_{p,1}}+\norm{M}{\B^s_{p,1}}\bigr).
\end{equation}
\end{lemma}

\begin{proof}
Use Definition~\ref{def:admissible}, Lemma~\ref{lem:standard-besov}(iv), and the fact that $\B^s_{p,1}$ is an algebra.  The condition $g(\bar a)=0$ is exactly what is needed for the zero-order term to belong to the homogeneous space.
\end{proof}

\begin{lemma}[Endpoint transport--diffusion estimate]\label{lem:endpoint-transport}
Let $1<p<\infty$, $s=d/p>0$, let $u$ be divergence free with
$u\in L^1(0,T;\dot B^{s+1}_{p,1})$, and assume
$\rho_0\in\dot B^s_{p,1}$ and $F\in L^1(0,T;\dot B^s_{p,1})$.  Let $\rho$ solve
\[
 \partial_t\rho+u\cdot\nabla\rho-\varepsilon\Delta\rho=F,
 \qquad \rho|_{t=0}=\rho_0,
\]
where $\varepsilon\ge0$.  Then, for $0\le t\le T$,
\begin{align}\label{eq:endpoint-transport}
 &\|\rho(t)\|_{\dot B^s_{p,1}}
 +c\varepsilon\|\rho\|_{L^1(0,t;\dot B^{s+2}_{p,1})}\notag\\
 &\qquad\le
 C\exp\left(C\int_0^t\|u(\tau)\|_{\dot B^{s+1}_{p,1}}\,d\tau\right)
 \left(
  \|\rho_0\|_{\dot B^s_{p,1}}
  +\|F\|_{L^1(0,t;\dot B^s_{p,1})}
 \right).
\end{align}
The solution obtained by smooth approximation belongs to $C_T\dot B^s_{p,1}$.
\end{lemma}

\begin{proof}
Apply $\dot\Delta_j$ to the equation.  The dissipative term gives the usual Bernstein lower bound
$\varepsilon 2^{2j}\|\dot\Delta_j\rho\|_{L^p}$.  The commutator estimate at the algebra endpoint reads
\[
 \sum_j2^{js}
 \|[u\cdot\nabla,\dot\Delta_j]\rho\|_{L^p}
 \lesssim
 \|u\|_{\dot B^{s+1}_{p,1}}\|\rho\|_{\dot B^s_{p,1}},
\]
which is the standard critical transport estimate for divergence-free velocities; see \cite[Chapter~3]{BCD} and \cite{DanchinTransport}.  Multiplying the dyadic inequality by $2^{js}$, summing in $j$, and applying Gronwall gives \eqref{eq:endpoint-transport}.  The time-continuity follows first for smooth approximations and then by summing the continuous dyadic blocks in $\ell^1$, using the estimate above to pass to the limit.
\end{proof}

\begin{proposition}[Local construction and uniform lifespan]\label{prop:local-input}
For every admissible quadruple and every datum in the phase space of Theorem~\ref{thm:abstract-main}, a smooth parabolic approximation yields a local solution satisfying
\begin{align}
 u&\in C_T\B^{s-1}_{p,1}\cap L^1_T\B^{s+1}_{p,1},
 &r&\in C_T\B^s_{p,1}(\mathcal W),\label{eq:local-class}\\
 u&\in L^2_T\B^s_{p,1},
 &P&\in L^1_T\B^s_{p,1},\label{eq:pressure-class}\\
 D_tu:=\partial_tu+u\cdot\nabla u
 &\in L^1_T\B^{s-1}_{p,1},
 &\partial_tr+u\cdot\nabla r&\in L^1_T\B^s_{p,1}.
 \label{eq:material-class}
\end{align}
If $(u_{0,n},r_{0,n})\to(u_0,r_0)$ in the critical phase space, then the corresponding approximations and limiting solutions have a common positive lifespan for all sufficiently large $n$.  On a possibly shorter common interval, the norms above are uniformly bounded and their time-integrable parts are uniformly absolutely continuous at the initial time.
\end{proposition}

\begin{proof}
The proof is standard but is included in \cref{sec:existence-appendix} to make the dependence on the constitutive law explicit.  It uses only Lemmas~\ref{lem:constitutive-tame} and~\ref{lem:endpoint-transport}, heat maximal regularity, and the low--high frequency construction of the lifespan; for classical maximal-regularity background see, for instance, \cite{Solonnikov,GigaSohr}.
\end{proof}

The pressure estimate used repeatedly below follows by applying divergence to the momentum equation:
\begin{equation}\label{eq:pressure-general}
 \norm{P}{L^1_T\B^s_{p,1}}
 \lesssim
 \norm{u}{L^2_T\B^s_{p,1}}^2
 +\norm{\mathscr S(r)}{L^1_T\B^s_{p,1}}.
\end{equation}

\section{The compensated transport product}\label{sec:null}

\begin{lemma}[Compensated product]\label{lem:null}
Let $2\le p<\infty$, $s=d/p>0$, $a\in\B^{s-1}_{p,1}(\R^d)^d$, and
$Y\in\B^{s+1}_{p,1}(\R^d)^d$.  The products below are understood
through the Bony decomposition, equivalently by first arguing for smooth
frequency-truncated fields and then passing to the limit by the estimates.
If $\diver a=0$, then
\begin{equation}\label{eq:null-low}
 \norm{a\cdot\nabla Y}{\B^{s-1}_{p,1}}
 \le C_0\norm{a}{\B^{s-1}_{p,1}}
       \norm{Y}{\B^{s+1}_{p,1}}.
\end{equation}
If $a,w\in\B^s_{p,1}$ and $\diver a=0$, then
\begin{equation}\label{eq:null-mid}
 \norm{a\cdot\nabla w}{\B^{s-1}_{p,1}}
 \le C_0\norm{a}{\B^s_{p,1}}\norm{w}{\B^s_{p,1}}.
\end{equation}
The same pointwise estimates hold in Bochner spaces, with time exponents combined by H\"older's inequality.  In particular,
\begin{align}
 \norm{a\cdot\nabla Y}{L^1_T\B^{s-1}_{p,1}}
 &\lesssim
 \norm{a}{L^1_T\B^{s-1}_{p,1}}
 \norm{Y}{L^\infty_T\B^{s+1}_{p,1}},\label{eq:null-time-1}\\
 \norm{a\cdot\nabla w}{L^1_T\B^{s-1}_{p,1}}
 &\lesssim
 \norm{a}{L^2_T\B^s_{p,1}}
 \norm{w}{L^2_T\B^s_{p,1}}.
 \label{eq:null-time-2}
\end{align}
\end{lemma}

\begin{proof}
It is enough to prove the bounds for smooth frequency-truncated vector
fields.  The estimates are stable under Littlewood--Paley truncation and
therefore define the bilinear operators by continuous extension in the
stated Besov spaces.  We first prove \eqref{eq:null-low}.  Set
\[
 A_j:=2^{j(s-1)}\norm{\dot\Delta_ja}{L^p},
 \qquad
 Y_j:=2^{j(s+1)}\norm{\dot\Delta_jY}{L^p}.
\]
Then $\sum_jA_j\simeq\norm{a}{\B^{s-1}_{p,1}}$ and
$\sum_jY_j\simeq\norm{Y}{\B^{s+1}_{p,1}}$.  We use Bony's
decomposition
\[
 a\cdot\nabla Y=\T_a\nabla Y+\T_{\nabla Y}a+\Rem(a,\nabla Y),
\]
where
\[
 \T_a\nabla Y=\sum_k \dot S_{k-1}a\cdot\nabla\dot\Delta_kY,\qquad
 \Rem(a,\nabla Y)=\sum_k\dot\Delta_ka\cdot
        \nabla\widetilde{\dot\Delta}_kY .
\]

Consider first $\T_a\nabla Y$.  By the frequency localization of the
paraproduct, for a fixed integer $N$,
\[
 \dot\Delta_j(\T_a\nabla Y)
 =\sum_{|k-j|\le N}
   \dot\Delta_j(\dot S_{k-1}a\cdot\nabla\dot\Delta_kY).
\]
Bernstein's inequality gives
\[
 \norm{\dot S_{k-1}a}{L^\infty}
 \le \sum_{\ell\le k-2}\norm{\dot\Delta_\ell a}{L^\infty}
 \lesssim
 \sum_{\ell\le k-2}2^{\ell s}\norm{\dot\Delta_\ell a}{L^p}
 =
 \sum_{\ell\le k-2}2^\ell A_\ell,
\]
while
\[
 \norm{\nabla\dot\Delta_kY}{L^p}
 \lesssim 2^k\norm{\dot\Delta_kY}{L^p}
 =2^{-ks}Y_k .
\]
Therefore
\[
\begin{aligned}
 2^{j(s-1)}\norm{\dot\Delta_j(\T_a\nabla Y)}{L^p}
 &\lesssim
 \sum_{|k-j|\le N}
 2^{j(s-1)}2^{-ks}Y_k
 \sum_{\ell\le k-2}2^\ell A_\ell  \\
 &\lesssim
 \sum_{|k-j|\le N}
 Y_k\sum_{\ell\le k-2}2^{\ell-k}A_\ell .
\end{aligned}
\]
After summing in $j$ and using the finite overlap in $|k-j|\le N$,
\[
\begin{aligned}
 \norm{\T_a\nabla Y}{\B^{s-1}_{p,1}}
 &\lesssim
 \sum_kY_k\sum_{\ell\le k-2}2^{\ell-k}A_\ell  \\
 &\lesssim
 \Bigl(\sum_kY_k\Bigr)\Bigl(\sum_\ell A_\ell\Bigr).
\end{aligned}
\]

For the second paraproduct,
\[
 \dot\Delta_j(\T_{\nabla Y}a)
 =\sum_{|k-j|\le N}
   \dot\Delta_j(\dot S_{k-1}\nabla Y\cdot\dot\Delta_ka).
\]
Again by Bernstein,
\[
\begin{aligned}
 \norm{\dot S_{k-1}\nabla Y}{L^\infty}
 &\le
 \sum_{\ell\le k-2}\norm{\nabla\dot\Delta_\ell Y}{L^\infty}  \\
 &\lesssim
 \sum_{\ell\le k-2}
 2^{\ell s}2^\ell\norm{\dot\Delta_\ell Y}{L^p}
 =
 \sum_{\ell\le k-2}Y_\ell .
\end{aligned}
\]
Since $\norm{\dot\Delta_ka}{L^p}=2^{-k(s-1)}A_k$, we get
\[
 2^{j(s-1)}\norm{\dot\Delta_j(\T_{\nabla Y}a)}{L^p}
 \lesssim
 \sum_{|k-j|\le N}A_k\sum_{\ell\le k-2}Y_\ell .
\]
Thus
\[
 \norm{\T_{\nabla Y}a}{\B^{s-1}_{p,1}}
 \lesssim
 \sum_k A_k\sum_{\ell\le k-2}Y_\ell
 \lesssim
 \Bigl(\sum_kA_k\Bigr)\Bigl(\sum_\ell Y_\ell\Bigr).
\]

It remains to estimate the remainder.  The divergence-free condition is
used exactly here.  Since $\diver\dot\Delta_ka=0$, componentwise
\[
 \dot\Delta_ka\cdot\nabla\widetilde{\dot\Delta}_kY
 =
 \diver\bigl(\widetilde{\dot\Delta}_kY\otimes\dot\Delta_ka\bigr).
\]
The frequency support of the remainder implies that only indices
$k\ge j-N_0$ contribute to $\dot\Delta_j\Rem(a,\nabla Y)$, for a fixed
integer $N_0$.  Since $2\le p<\infty$, Bernstein's inequality from $L^{p/2}$ to
$L^p$ gives
\[
\begin{aligned}
 &2^{j(s-1)}\norm{\dot\Delta_j\Rem(a,\nabla Y)}{L^p}  \\
 &\quad\lesssim
 2^{j(s-1)}2^j2^{jd/p}
 \sum_{k\ge j-N_0}
 \norm{\widetilde{\dot\Delta}_kY\otimes\dot\Delta_ka}{L^{p/2}} .
\end{aligned}
\]
Here $d/p=s$.  By H\"older's inequality and the definitions of $A_k$ and
$Y_k$,
\[
 \norm{\widetilde{\dot\Delta}_kY\otimes\dot\Delta_ka}{L^{p/2}}
 \lesssim
 2^{-k(s+1)}Y_k\,2^{-k(s-1)}A_k,
\]
where, as usual, $Y_k$ may be replaced by a finite sum of neighboring
dyadic coefficients without changing its $\ell^1$ size.  Consequently,
\[
 2^{j(s-1)}\norm{\dot\Delta_j\Rem(a,\nabla Y)}{L^p}
 \lesssim
 \sum_{k\ge j-N_0}2^{-2s(k-j)}A_kY_k .
\]
Since $s>0$,
\[
\begin{aligned}
 \norm{\Rem(a,\nabla Y)}{\B^{s-1}_{p,1}}
 &\lesssim
 \sum_j\sum_{k\ge j-N_0}2^{-2s(k-j)}A_kY_k  \\
 &\lesssim
 \sum_k A_kY_k
 \lesssim
 \Bigl(\sum_kA_k\Bigr)\Bigl(\sum_kY_k\Bigr).
\end{aligned}
\]
This is the only place where the solenoidal structure is used: it places
one derivative on the low output of the high--high interaction and produces
the summable kernel $2^{-2s(k-j)}$.
Combining the three estimates proves \eqref{eq:null-low}.

For \eqref{eq:null-mid}, since $\diver a=0$ and the divergence is row-wise,
\[
 a\cdot\nabla w=\diver(w\otimes a),
\]
and $\B^s_{p,1}$ is an algebra.  Hence
\[
 \norm{a\cdot\nabla w}{\B^{s-1}_{p,1}}
 \lesssim
 \norm{w\otimes a}{\B^s_{p,1}}
 \lesssim
 \norm{a}{\B^s_{p,1}}\norm{w}{\B^s_{p,1}},
\]
which proves \eqref{eq:null-mid}.  The time-dependent versions follow
from the identical dyadic estimates and H\"older's inequality in time.
\end{proof}

\begin{remark}\label{rem:not-generic}
The estimate is not a disguised generic multiplier theorem.  For $s<1/2$, the product
$\B^s_{p,1}\cdot\B^{s-1}_{p,1}$ is not generally contained in
$\B^{s-1}_{p,1}$.  The derivative and the condition $\diver a=0$ are both essential.  Without the solenoidal rewriting of the high--high term as a divergence, the high-to-low cascade would contain the factor $2^{(1-2s)(k-j)}$, which is not summable when $s<1/2$.  The endpoint $p=\infty$ corresponds formally to $s=0$ under $s=d/p$; then the compensated high--high kernel above is no longer summable, which is why the argument does not include that endpoint.
\end{remark}

\begin{proposition}[Generic product failure without the solenoidal structure]\label{prop:generic-failure}
Let $1<p<\infty$ and $0<s<1/2$.
\begin{enumerate}[label=\textup{(\roman*)}]
 \item There are real-valued Schwartz functions $f_N,g_N$ such that
 \[
  \sup_N\|f_N\|_{\dot B^s_{p,1}}+
  \sup_N\|g_N\|_{\dot B^{s-1}_{p,1}}<\infty,
 \]
 but $\|f_Ng_N\|_{\dot B^{s-1}_{p,1}}\to\infty$.
 \item There are real-valued Schwartz vector fields $a_N,Y_N$ such that
 \[
  \sup_N\|a_N\|_{\dot B^{s-1}_{p,1}}+
  \sup_N\|Y_N\|_{\dot B^{s+1}_{p,1}}<\infty,
 \]
 while
 \[
  \|(a_N\cdot\nabla)Y_N\|_{\dot B^{s-1}_{p,1}}\to\infty.
 \]
 In this construction $a_N$ is not divergence free.  Thus the derivative alone does not recover the estimate: the solenoidal condition in Lemma~\ref{lem:null} is essential.
\end{enumerate}
\end{proposition}

\begin{proof}
Choose a real-valued $\chi\in\mathcal S$, $\chi\not\equiv0$, whose Fourier
transform is supported in a sufficiently small ball around the origin.  Fix
nonzero vectors $\xi,\eta$ with $\xi_1\ne0$ and choose $\eta$ so that the
Fourier supports of $\cos(\eta\cdot x)\chi^2$ and
$\sin(\eta\cdot x)\chi^2$ are contained in a fixed annulus away from the
origin.  Put
\[
 f_N(x)=2^{-Ns}\cos(2^N\xi\cdot x)\chi(x),
 \qquad
 g_N(x)=2^{N(1-s)}\cos((2^N\xi-\eta)\cdot x)\chi(x).
\]
Frequency localization gives
$\|f_N\|_{\dot B^s_{p,1}}\simeq1$ and
$\|g_N\|_{\dot B^{s-1}_{p,1}}\simeq1$.  On the other hand,
\[
 f_Ng_N=\frac12 2^{N(1-2s)}
 \bigl(\cos(\eta\cdot x)+\cos((2^{N+1}\xi-\eta)\cdot x)\bigr)\chi^2.
\]
For $N$ large, the fixed annular part and the high-frequency part have
disjoint Fourier supports.  Applying the fixed finite sum of dyadic blocks
which selects the annulus containing $\cos(\eta\cdot x)\chi^2$ gives
\[
 \|f_Ng_N\|_{\dot B^{s-1}_{p,1}}
 \gtrsim 2^{N(1-2s)}.
\]
This proves the first assertion because $s<1/2$.

For the second assertion, take
\[
 a_N=g_N e_1,\qquad
 Y_N=2^{-N(s+1)}\cos(2^N\xi\cdot x)\chi(x)e_1.
\]
Then $\|a_N\|_{\dot B^{s-1}_{p,1}}+\|Y_N\|_{\dot B^{s+1}_{p,1}}\lesssim1$.
Since
\[
 \partial_1\bigl(2^{-N(s+1)}\cos(2^N\xi\cdot x)\chi\bigr)
 =-\xi_1 2^{-Ns}\sin(2^N\xi\cdot x)\chi
   +2^{-N(s+1)}\cos(2^N\xi\cdot x)\partial_1\chi,
\]
the leading low-frequency part of $(a_N\cdot\nabla)Y_N$ is
\[
 -\frac{\xi_1}{2}2^{N(1-2s)}\sin(\eta\cdot x)\chi^2 e_1.
\]
The remaining leading oscillation is supported near frequency $2^{N+1}\xi$,
and the terms in which the derivative falls on $\chi$ are smaller by a
factor $2^{-N}$ at the fixed low frequency.  The same fixed dyadic
projection therefore gives
\[
 \|(a_N\cdot\nabla)Y_N\|_{\dot B^{s-1}_{p,1}}
 \gtrsim 2^{N(1-2s)},
\]
which tends to infinity.  The vector field $a_N$ is not divergence free,
and the claim follows.
\end{proof}

\begin{remark}[The endpoint $p=2d$]\label{rem:threshold-endpoint}
Proposition~\ref{prop:generic-failure} concerns the strict range $s<1/2$; under the standing relation $s=d/p$ used in the applications, this is the high-$p$ range $p>2d$.  At $s=1/2$ the standard generic product theorem is at its limiting balance, and the single-frequency construction above does not produce divergence.  None of the positive arguments in this paper requires a generic endpoint failure: the compensated estimate and the Piola graph treat $p=2d$ and $p>2d$ uniformly.  Accordingly, whenever generic-product failure is asserted below, it refers to the strict high-$p$ range.
\end{remark}

\section{A rigorous compensated Piola pullback}\label{sec:piola}

For $Y\in\B^{s+1}_{p,1}$ and a solenoidal
$a\in\B^{s-1}_{p,1}$, define
\[
 \K_Ya:=(a\cdot\nabla)Y,
 \qquad L_Ya:=a+\K_Ya,
\]
where the product is the compensated transport product of
Lemma~\ref{lem:null}.  At middle regularity, for
$a\in\B^s_{p,1}$, the same notation is used and agrees with the usual
Bony product.

\begin{lemma}[Coercivity on divergence-free fields]\label{lem:coercive}
Let $C_0$ be the constant in \eqref{eq:null-low}.  There exists
$\varepsilon_0>0$, depending only on $d$ and $p$, with the following
property.  If $Y\in\B^{s+1}_{p,1}$ satisfies
\begin{equation}\label{eq:Y-small-static}
 \norm{Y}{\B^{s+1}_{p,1}}\le\varepsilon_0,
\end{equation}
then
\begin{equation}\label{eq:coercive}
 \norm{a}{\B^{s-1}_{p,1}}
 \le2\norm{L_Ya}{\B^{s-1}_{p,1}}
 \qquad\text{for every }a\in\B^{s-1}_{p,1},\quad\diver a=0.
\end{equation}
Moreover, for the same $Y$,
\begin{equation}\label{eq:coercive-mid}
 \norm{a}{\B^s_{p,1}}
 \le2\norm{L_Ya}{\B^s_{p,1}}
 \qquad\text{for every }a\in\B^s_{p,1},\quad\diver a=0.
\end{equation}
\end{lemma}

\begin{proof}
Choose a constant $C_1$, depending only on $d$ and $p$, such that
\[
 \norm{(a\cdot\nabla)Y}{\B^s_{p,1}}
 \le C_1\norm{a}{\B^s_{p,1}}\norm{Y}{\B^{s+1}_{p,1}},
 \qquad a\in\B^s_{p,1}.
\]
This estimate follows from the algebra property of $\B^s_{p,1}$ and
the equivalence
$\norm{\nabla Y}{\B^s_{p,1}}\simeq\norm{Y}{\B^{s+1}_{p,1}}$.
Fix
\[
 \varepsilon_0
 \le \min\Bigl\{\frac1{2C_0},\frac1{2C_1}\Bigr\}.
\]

Let first $a\in\B^{s-1}_{p,1}$ with $\diver a=0$.  Since
$L_Ya=a+(a\cdot\nabla)Y$,
\[
 a=L_Ya-(a\cdot\nabla)Y.
\]
Taking the $\B^{s-1}_{p,1}$ norm and using \eqref{eq:null-low}, we get
\[
\begin{aligned}
 \norm{a}{\B^{s-1}_{p,1}}
 &\le
 \norm{L_Ya}{\B^{s-1}_{p,1}}
 +\norm{(a\cdot\nabla)Y}{\B^{s-1}_{p,1}}  \\
 &\le
 \norm{L_Ya}{\B^{s-1}_{p,1}}
 +C_0\norm{Y}{\B^{s+1}_{p,1}}
       \norm{a}{\B^{s-1}_{p,1}} .
\end{aligned}
\]
Under \eqref{eq:Y-small-static}, the choice of $\varepsilon_0$ gives
$C_0\norm{Y}{\B^{s+1}_{p,1}}\le1/2$.  Hence
\[
 \frac12\norm{a}{\B^{s-1}_{p,1}}
 \le \norm{L_Ya}{\B^{s-1}_{p,1}},
\]
which proves \eqref{eq:coercive}.

The proof at regularity $s$ is identical, with the preceding algebra
estimate replacing \eqref{eq:null-low}.  Indeed,
\[
\begin{aligned}
 \norm{a}{\B^s_{p,1}}
 &\le
 \norm{L_Ya}{\B^s_{p,1}}
 +\norm{(a\cdot\nabla)Y}{\B^s_{p,1}}  \\
 &\le
 \norm{L_Ya}{\B^s_{p,1}}
 +C_1\norm{Y}{\B^{s+1}_{p,1}}\norm{a}{\B^s_{p,1}}.
\end{aligned}
\]
Since $C_1\norm{Y}{\B^{s+1}_{p,1}}\le1/2$, moving the last term to the
left gives \eqref{eq:coercive-mid}.
\end{proof}

The next proposition is the point that removes the circular definition $z=Av$ at negative regularity.

\begin{proposition}[Closed Piola graph]\label{prop:closed-piola}
Let $Y_n,Y\in C([0,T];\B^{s+1}_{p,1})$ and
$v_n,v\in C_T\B^{s-1}_{p,1}\cap L^2_T\B^s_{p,1}$ satisfy
\[
 Y_n\to Y\quad\text{in }C_T\B^{s+1}_{p,1},
 \qquad
 v_n\to v\quad\text{in }C_T\B^{s-1}_{p,1}
 \cap L^2_T\B^s_{p,1}.
\]
Assume, with $\varepsilon_0$ from Lemma~\ref{lem:coercive}, that
\[
 \sup_n \norm{Y_n}{L^\infty_T\B^{s+1}_{p,1}}\le\varepsilon_0
\]
and that there are smooth divergence-free fields
\[
 z_n\in C_T\B^{s-1}_{p,1}\cap L^2_T\B^s_{p,1}
\]
such that
\begin{equation}\label{eq:piola-approx-identity}
 v_n=z_n+(z_n\cdot\nabla)Y_n.
\end{equation}
Then there is a unique
\[
 z\in C_T\B^{s-1}_{p,1}
 \cap L^2_T\B^s_{p,1},
 \qquad\diver z=0,
\]
such that
\begin{equation}\label{eq:piola-limit-identity}
 v=z+(z\cdot\nabla)Y.
\end{equation}
Moreover,
\begin{align}
 \norm{z}{L^\infty_T\B^{s-1}_{p,1}}
 &\le2\norm{v}{L^\infty_T\B^{s-1}_{p,1}},\label{eq:piola-low-bound}\\
 \norm{z}{L^2_T\B^s_{p,1}}
 &\le2\norm{v}{L^2_T\B^s_{p,1}}.\label{eq:piola-mid-bound}
\end{align}
\end{proposition}

\begin{proof}
Throughout this proof set
\[
 E:=\B^{s-1}_{p,1}(\R^d)^d,
 \qquad
 E_\sigma:=\{a\in E:\diver a=0\}.
\]
We first obtain estimates for the approximating fields.  The dyadic
proof of Lemma~\ref{lem:coercive}, with the time exponents inserted by
H\"older's inequality, gives the Bochner-in-time versions
\[
 \norm{a}{L^\infty_T\B^{s-1}_{p,1}}
 \le
 2\norm{L_{Y_n}a}{L^\infty_T\B^{s-1}_{p,1}},
 \qquad
 \norm{a}{L^2_T\B^s_{p,1}}
 \le
 2\norm{L_{Y_n}a}{L^2_T\B^s_{p,1}},
\]
for divergence-free $a$, uniformly in $n$.  Applying this to
$a=z_n$ and using $L_{Y_n}z_n=v_n$ yields
\begin{align*}
 \norm{z_n}{L^\infty_T\B^{s-1}_{p,1}}
 &\le2\norm{v_n}{L^\infty_T\B^{s-1}_{p,1}},\\
 \norm{z_n}{L^2_T\B^s_{p,1}}
 &\le2\norm{v_n}{L^2_T\B^s_{p,1}}.
\end{align*}
Since $v_n\to v$ in the corresponding spaces, the sequence $(z_n)$ is
bounded in
$L^\infty_T\B^{s-1}_{p,1}\cap L^2_T\B^s_{p,1}$.

We spell out the compactness step because the Besov summability index is $1$ and the ambient space is not reflexive.  For every dyadic index $j$, the sequence $\dot\Delta_jz_n$ is bounded in
\[
 L^\infty(0,T;L^p)\cap L^2(0,T;L^p).
\]
Since $1<p<\infty$, Banach--Alaoglu and reflexivity give, after a diagonal extraction, weak-star convergence in the first space and weak convergence in the second for every $j$.  The compatibility of the Littlewood--Paley projectors identifies the limiting blocks with a distribution $z$ on $(0,T)\times\R^d$: after testing against a Schwartz function and using the finite overlap of the dyadic supports, the blockwise limits reconstruct a unique element of $\mathcal S'_h$ whose dyadic blocks are precisely the limits just obtained.  For every finite set of indices $J\subset\mathbb Z$, weak lower semicontinuity of the convex functionals
\[
 (f_j)_{j\in J}\longmapsto
 \left\|\sum_{j\in J}2^{j(s-1)}\|f_j(t)\|_{L^p}\right\|_{L^\infty_t},
\]
\[
 (f_j)_{j\in J}\longmapsto
 \left\|\sum_{j\in J}2^{js}\|f_j(t)\|_{L^p}\right\|_{L^2_t}
\]
and then monotone convergence as $J\uparrow\mathbb Z$--equivalently Lemma~\ref{lem:bochner-besov-fatou}--yield the Bochner--Besov Fatou bounds
\[
 z\in L^\infty_T\B^{s-1}_{p,1}\cap L^2_T\B^s_{p,1},
\]
with the corresponding norms bounded by the lower limits of the norms of $z_n$.  The distributional limit is solenoidal because each $z_n$ is.  Using the strong convergence of $v_n$ in the preceding estimates gives \eqref{eq:piola-low-bound} and \eqref{eq:piola-mid-bound}.

We next pass to the nonlinear term.  Let
\[
 Y^{(N)}:=\sum_{|j|\le N}\dot\Delta_jY .
\]
Then $Y^{(N)}\to Y$ in $C_T\B^{s+1}_{p,1}$.  We decompose
\[
\begin{aligned}
 (z_n\cdot\nabla)Y_n-(z\cdot\nabla)Y
 ={}&z_n\cdot\nabla(Y_n-Y)
    +z_n\cdot\nabla(Y-Y^{(N)})\\
 &+(z_n-z)\cdot\nabla Y^{(N)}
    +z\cdot\nabla(Y^{(N)}-Y).
\end{aligned}
\]
The first term tends to zero in $L^\infty_T\B^{s-1}_{p,1}$ because
Lemma~\ref{lem:null} gives
\[
 \norm{z_n\cdot\nabla(Y_n-Y)}{L^\infty_T\B^{s-1}_{p,1}}
 \lesssim
 \norm{z_n}{L^\infty_T\B^{s-1}_{p,1}}
 \norm{Y_n-Y}{L^\infty_T\B^{s+1}_{p,1}}.
\]
The second and fourth terms are uniformly small as $N\to\infty$:
\[
\begin{aligned}
 \norm{z_n\cdot\nabla(Y-Y^{(N)})}{L^\infty_T\B^{s-1}_{p,1}}
 &\lesssim
 \norm{z_n}{L^\infty_T\B^{s-1}_{p,1}}
 \norm{Y-Y^{(N)}}{L^\infty_T\B^{s+1}_{p,1}},\\
 \norm{z\cdot\nabla(Y^{(N)}-Y)}{L^\infty_T\B^{s-1}_{p,1}}
 &\lesssim
 \norm{z}{L^\infty_T\B^{s-1}_{p,1}}
 \norm{Y^{(N)}-Y}{L^\infty_T\B^{s+1}_{p,1}} .
\end{aligned}
\]
For fixed $N$, the coefficient $\nabla Y^{(N)}$ is smooth in $x$ with
all relevant norms bounded uniformly in time; multiplication by this
coefficient is therefore continuous for distributional convergence.
Thus
\[
 (z_n-z)\cdot\nabla Y^{(N)}\longrightarrow0
 \quad\text{in distributions}
\]
for each fixed $N$.  Letting first $n\to\infty$ and then
$N\to\infty$ yields
\[
 (z_n\cdot\nabla)Y_n\longrightarrow(z\cdot\nabla)Y
 \quad\text{in distributions}.
\]
Since also $v_n\to v$, passing to the limit in
\eqref{eq:piola-approx-identity} gives \eqref{eq:piola-limit-identity}.

The solution of the limiting identity is unique.  Indeed, if $z$ and
$\widetilde z$ both satisfy \eqref{eq:piola-limit-identity}, then
$\delta z:=z-\widetilde z$ is divergence free and
$L_Y\delta z=0$.  Lemma~\ref{lem:coercive} gives
\[
 \norm{\delta z}{\B^{s-1}_{p,1}}
 \le2\norm{L_Y\delta z}{\B^{s-1}_{p,1}}=0.
\]
Hence $z=\widetilde z$.  Consequently every distributional subsequential
limit of $(z_n)$ is the same, so the whole approximation has this limit.

It remains to identify the continuous representative.  The preceding
limit gives \eqref{eq:piola-limit-identity} as a spacetime distribution
identity.  Since $v\in C_T\dot B^{s-1}_{p,1}$ and
$Y\in C_T\dot B^{s+1}_{p,1}$, the map
$t\mapsto \K_{Y(t)}a$ is continuous from $[0,T]$ to $E$ for each fixed
$a\in E_\sigma$.  Hence the spacetime identity implies that
$L_{Y(t)}z(t)=v(t)$ for almost every $t$, after replacing $z$ by its
Lebesgue representative in $L^\infty_T E$.

Let $\mathcal T$ be the full-measure set of such times.  For
$t,t_0\in\mathcal T$, subtract the two identities and write
\[
 L_{Y(t)}\bigl(z(t)-z(t_0)\bigr)
 =
 v(t)-v(t_0)-\K_{Y(t)-Y(t_0)}z(t_0).
\]
Coercivity and Lemma~\ref{lem:null} give
\begin{align*}
 \norm{z(t)-z(t_0)}{\B^{s-1}_{p,1}}
 \lesssim{}&\norm{v(t)-v(t_0)}{\B^{s-1}_{p,1}}\\
 &+\norm{z(t_0)}{\B^{s-1}_{p,1}}
 \norm{Y(t)-Y(t_0)}{\B^{s+1}_{p,1}}.
\end{align*}
Thus $z|_{\mathcal T}$ is continuous at every point of $\mathcal T$.
For any $t_*\in[0,T]$, choose $t_m\in\mathcal T$ with $t_m\to t_*$;
the estimate above shows that $z(t_m)$ is Cauchy in $E$ and that its
limit is independent of the approximating sequence.  Denote the limit by
$\bar z(t_*)$.  Then $\bar z\in C_T E_\sigma$, $\bar z=z$ almost
everywhere, and the $L^2_T\dot B^s_{p,1}$ bound is unchanged.  Passing
to the limit in $L_{Y(t_m)}z(t_m)=v(t_m)$ gives
$L_{Y(t_*)}\bar z(t_*)=v(t_*)$ for every $t_*$.  Replacing $z$ by
$\bar z$ gives the asserted continuous solution.
\end{proof}

\begin{corollary}[Piola pullback of a critical flow]\label{cor:flow-piola}
Let $u\in C_T\B^{s-1}_{p,1}\cap L^1_T\B^{s+1}_{p,1}$ be divergence free, let $X$ be its volume-preserving bi-Lipschitz flow, and put
$Y=X-\Id$.  Then $Y\in C_T\B^{s+1}_{p,1}$; set
$F:=I+\D Y$, the distributional derivative of $X$.  If, with $\varepsilon_0$ from
Lemma~\ref{lem:coercive},
\[
 \norm{Y}{L^\infty_T\B^{s+1}_{p,1}}\le\frac{\varepsilon_0}{2},
\]
then there is a unique field $z$ in the class of Proposition~\ref{prop:closed-piola} such that
\begin{equation}\label{eq:vz-rigorous}
 \diver z=0,
 \qquad v:=u\circ X=Fz=z+(z\cdot\nabla)Y.
\end{equation}
For smooth $u$, this field is $z=F^{-1}v$.  For a critical solution, \eqref{eq:vz-rigorous} is the definition of its Piola pullback.
\end{corollary}

\begin{proof}
Choose smooth, rapidly decaying, divergence-free approximations $u_n$ such that
\[
 u_n\longrightarrow u
 \quad\hbox{in }C_T\B^{s-1}_{p,1}
 \cap L^1_T\B^{s+1}_{p,1}.
\]
By the interpolation estimate of Lemma~\ref{lem:standard-besov},
$u_n\to u$ in $L^2_T\B^s_{p,1}$.  In particular,
$u_n-u\to0$ in $L^1_T\B^s_{p,1}\hookrightarrow L^1_TL^\infty$.
Let $X_n$ be the smooth flow generated by $u_n$, and set
$Y_n:=X_n-\Id$.  Since
$\nabla u_n,\nabla u\in L^1_T\B^s_{p,1}\hookrightarrow L^1_TL^\infty$, the flows $X_n$, $X$ and their inverses have uniform bi-Lipschitz bounds.  The flow equation and Gronwall's inequality give
\[
 X_n\longrightarrow X
 \quad\hbox{uniformly on }[0,T]\times\R^d.
\]
All these flows are volume preserving because the velocities are divergence free.

We next construct the derivative of the critical flow by approximation, rather than differentiating $X$ classically.  For the smooth flows set
\[
 G_n:=\D Y_n=\D X_n-I,
 \qquad
 H_n:=(\nabla u_n)\circ X_n .
\]
Then
\begin{equation}\label{eq:smooth-flow-gradient-integral}
 G_n(t)=\int_0^t H_n(\tau)\bigl(I+G_n(\tau)\bigr)\,d\tau .
\end{equation}
The composition lemma, at index $s$, gives
\[
 H_n\longrightarrow H:=(\nabla u)\circ X
 \quad\hbox{in }L^1_T\B^s_{p,1};
\]
indeed, split the difference into
$((\nabla u_n-\nabla u)\circ X_n)$ and
$((\nabla u)\circ X_n-(\nabla u)\circ X)$, using the uniform
composition bound for the first term and the strong composition
convergence for the second.  Since $\B^s_{p,1}$ is an algebra, \eqref{eq:smooth-flow-gradient-integral} and Gronwall's inequality show that $(G_n)$ is Cauchy in $C_T\B^s_{p,1}$.  Let
$G$ be its limit.  Since $X_n\to X$ uniformly, $Y_n\to Y$ in distributions, and therefore $G=\D Y$ distributionally.  Thus
\[
 Y\in C_T\B^{s+1}_{p,1},
 \qquad
 Y_n\longrightarrow Y
 \quad\hbox{in }C_T\B^{s+1}_{p,1},
\]
and $F:=I+\D Y$ is the distributional derivative of the flow.  Passing to the limit in the integral equation gives
\[
 \D Y(t)=\int_0^t ((\nabla u)\circ X)(\tau)\bigl(I+\D Y(\tau)\bigr)\,d\tau
 \quad\hbox{in }C_T\B^s_{p,1}.
\]
The strict margin in the smallness assumption then implies, after
discarding finitely many indices, that
\[
 \sup_n\norm{Y_n}{L^\infty_T\B^{s+1}_{p,1}}\le\varepsilon_0.
\]

Now put
\[
 v_n:=u_n\circ X_n,
 \qquad
 v:=u\circ X .
\]
We claim that $v_n\to v$ in the two spaces appearing in
Proposition~\ref{prop:closed-piola}.  Write
\[
 v_n-v=(u_n-u)\circ X_n+\bigl(u\circ X_n-u\circ X\bigr).
\]
The first term tends to zero by the uniform composition bounds at
indices $s-1$ and $s$.  The second term tends to zero by the strong composition part of Lemma~\ref{lem:composition}, first in
$C_T\B^{s-1}_{p,1}$ and then by its Bochner $L^2_t$ assertion in $L^2_T\B^s_{p,1}$.  Hence
\[
 v_n\longrightarrow v
 \quad\hbox{in }C_T\B^{s-1}_{p,1}
 \cap L^2_T\B^s_{p,1}.
\]

For each smooth flow, define
\[
 F_n:=\D X_n,
 \qquad A_n:=F_n^{-1},
 \qquad z_n:=A_nv_n .
\]
For the smooth approximations all products are classical, and the fields $z_n$ are smooth, divergence free, and belong to the Besov classes required in Proposition~\ref{prop:closed-piola}.  Since $u_n$ is divergence free, $X_n$ is volume preserving.  The classical Piola identity therefore gives
\[
 \diver z_n=0.
\]
Moreover,
\[
 v_n=F_nz_n=(I+\D Y_n)z_n
     =z_n+(z_n\cdot\nabla)Y_n.
\]
Equivalently, Lemma~\ref{lem:coercive} applied to this smooth identity yields the same uniform Besov bounds for $z_n$ as those required by Proposition~\ref{prop:closed-piola}.
Thus the hypotheses of Proposition~\ref{prop:closed-piola} are
satisfied.  We obtain a unique
\[
 z\in C_T\B^{s-1}_{p,1}\cap L^2_T\B^s_{p,1},
 \qquad \diver z=0,
\]
such that
\[
 v=z+(z\cdot\nabla)Y=Fz.
\]
This proves \eqref{eq:vz-rigorous}.

If $u$ is smooth, the classical field $F^{-1}v$ is divergence free by
the Piola identity and satisfies the same equation $v=Fz$.  The
uniqueness in Proposition~\ref{prop:closed-piola} therefore identifies
the above limiting field with $F^{-1}v$.  This construction never forms
the undefined generic product $(F^{-1}-I)v$ at regularity $s-1$.
\end{proof}

\begin{proposition}[Critical volume-preserving flows and affine Jacobians]\label{prop:critical-flow-jacobian}
Let $u\in C_T\B^{s-1}_{p,1}\cap L^1_T\B^{s+1}_{p,1}$ be divergence free, and let $X=\Id+Y$ be its flow.  Then $X$ is a volume-preserving bi-Lipschitz homeomorphism,
\[
 Y\in C_T\B^{s+1}_{p,1},\qquad F:=I+\D Y=\D X,
 \qquad F-I\in C_T\B^s_{p,1},
\]
where the derivative is understood distributionally.  Moreover
\[
 \det F=1
\]
in the affine homogeneous Besov algebra, and
\[
 A:=F^{-1}=(\operatorname{cof}F)^\top=I+(A-I),
 \qquad A-I\in C_T\B^s_{p,1}.
\]
The Piola identity
\begin{equation}\label{eq:critical-piola-identity}
 \partial_\alpha A_{\alpha\beta}=0
\end{equation}
holds in distributions.  If $u_n\to u$ in
$C_T\B^{s-1}_{p,1}\cap L^1_T\B^{s+1}_{p,1}$ with smooth divergence-free $u_n$, then the corresponding flows satisfy
\[
 Y_n\to Y\quad\hbox{in }C_T\B^{s+1}_{p,1},\qquad
 F_n-I\to F-I,\quad A_n-I\to A-I
 \quad\hbox{in }C_T\B^s_{p,1},
\]
and the identities above are the limits of the classical identities for $X_n$.
\end{proposition}

\begin{proof}
The flow and its inverse are bi-Lipschitz because $\nabla u\in L^1_TL^\infty$, and the standard ODE stability argument gives uniform convergence of smooth flows $X_n$ to $X$ whenever $u_n\to u$ in $L^1_TL^\infty$.  Volume preservation follows by passing to the limit from the smooth divergence-free flows.  The construction of $Y$ and the convergence $Y_n\to Y$ in
$C_T\B^{s+1}_{p,1}$ were obtained in the first part of the proof of Corollary~\ref{cor:flow-piola}; this part does not use the smallness assumption in that corollary.  In particular $F_n-I=\D Y_n\to F-I$ in $C_T\B^s_{p,1}$.  For each smooth divergence-free approximation,
$\det F_n=1$, $A_n=F_n^{-1}=(\operatorname{cof}F_n)^\top$, and
$\partial_\alpha(A_n)_{\alpha\beta}=0$ classically.

Write $F=I+G$ and $F_n=I+G_n$.  The quantities
$\det(I+G)-1$, $(\operatorname{cof}(I+G))^\top-I$, and their smooth analogues are finite sums of monomials with at least one factor $G$ or $G_n$.  Since $\B^s_{p,1}$ is a Banach algebra, these polynomial maps are continuous from $\B^s_{p,1}$ to itself on bounded sets.  Hence
\[
 \det F_n-1\to \det F-1,
 \qquad
 A_n-I=(\operatorname{cof}F_n)^\top-I\to (\operatorname{cof}F)^\top-I
\]
in $C_T\B^s_{p,1}$.  The first limit and $\det F_n=1$ give
$\det F=1$.  The second defines $A-I\in C_T\B^s_{p,1}$ and gives
$A=(\operatorname{cof}F)^\top$.  Since $\det F=1$, the affine matrix identity
$(\operatorname{cof}F)^\top F=(\det F)I$ gives $AF=FA=I$ in the Besov algebra, so $A=F^{-1}$.  Finally, passing to the distributional limit in
$\partial_\alpha(A_n)_{\alpha\beta}=0$ gives \eqref{eq:critical-piola-identity}.
\end{proof}

We also need a rigorous time-differentiation statement.

\begin{lemma}[Differentiating a compensated Piola identity]\label{lem:time-diff}
Let
\[
 E:=\B^{s-1}_{p,1}(\R^d)^d,
 \qquad E_\sigma:=\{a\in E:\diver a=0\}.
\]
Assume
\[
 Y\in W^{1,1}(0,T;\B^{s+1}_{p,1}),\qquad
 z\in C([0,T];E_\sigma)\cap L^2_T\B^s_{p,1},
\]
and that the smallness condition \eqref{eq:Y-small-static} holds uniformly in time.  If
\[
 w:=L_Yz\in W^{1,1}(0,T;E),
 \qquad \partial_tY\in L^2_T\B^s_{p,1},
\]
then $z\in W^{1,1}(0,T;E_\sigma)$ and, almost everywhere,
\begin{equation}\label{eq:time-diff-identity}
 \partial_tw=L_Y(\partial_tz)+(z\cdot\nabla)\partial_tY.
\end{equation}
Moreover,
\begin{equation}\label{eq:time-diff-est}
 \norm{\partial_tz}{L^1_TE}
 \lesssim
 \norm{\partial_tw}{L^1_TE}
 +\norm{z}{L^2_T\B^s_{p,1}}
  \norm{\partial_tY}{L^2_T\B^s_{p,1}}.
\end{equation}
\end{lemma}

\begin{proof}
Set $K(t)a=(a\cdot\nabla)Y(t)$ on $E_\sigma$.  Lemma~\ref{lem:null} gives
\[
 \|K(t)\|_{\mathcal L(E_\sigma,E)}
 \le C\|Y(t)\|_{\dot B^{s+1}_{p,1}}.
\]
Because $Y\in W^{1,1}(0,T;\dot B^{s+1}_{p,1})$, its time derivative
also belongs to $L^1_T\dot B^{s+1}_{p,1}$.  Hence, for almost every
$t$,
\[
 K'(t)a=(a\cdot\nabla)\partial_tY(t),
 \qquad
 \|K'(t)\|_{\mathcal L(E_\sigma,E)}
 \le C\|\partial_tY(t)\|_{\dot B^{s+1}_{p,1}}.
\]
Consequently $L(t):=I+K(t)$ belongs to
$W^{1,1}(0,T;\mathcal L(E_\sigma,E))$.  The additional assumption
$\partial_tY\in L^2_T\dot B^s_{p,1}$ is not needed for this operator
absolute continuity; it is used below to estimate
$(z\cdot\nabla)\partial_tY$ by the middle compensated bound.

We first prove that $z$ is an $E_\sigma$-valued BV function.  For
$0<h<T$ and $0<t<T-h$, subtraction at times $t+h$ and $t$ gives the exact identity
\[
 \delta_hw(t)
 =L_{Y(t+h)}\delta_hz(t)
 +(z(t)\cdot\nabla)\delta_hY(t),
\]
where $\delta_hf(t)=h^{-1}(f(t+h)-f(t))$.  Coercivity at time $t+h$ therefore yields
\[
 \|\delta_hz(t)\|_E
 \lesssim
 \|\delta_hw(t)\|_E
 +\|(z(t)\cdot\nabla)\delta_hY(t)\|_E.
\]
The first term is uniformly integrable because $w\in W^{1,1}(0,T;E)$.  For the second, use the middle-regularity compensated estimate and Cauchy--Schwarz in time:
\[
 \int_0^{T-h}\|(z(t)\cdot\nabla)\delta_hY(t)\|_E\,dt
 \lesssim
 \|z\|_{L^2_T\dot B^s_{p,1}}
 \|\delta_hY\|_{L^2(0,T-h;\dot B^s_{p,1})}.
\]
The last factor is bounded by
$\|\partial_tY\|_{L^2_T\dot B^s_{p,1}}$.  Hence the difference quotients of $z$ are uniformly bounded in $L^1(0,T-h;E)$.  By the standard Banach-valued BV criterion, equivalently by testing against smooth $E'$-valued functions and using duality, this implies $z\in BV(0,T;E_\sigma)$.  The difference quotients are solenoidal, so the distributional derivative $Dz$ takes values in the closed subspace $E_\sigma$.

The space $E$ has the Radon--Nikodym property, as proved in
\cref{sec:bv-appendix}; the closed subspace $E_\sigma$ has the same property.  Thus $Dz$ admits the decomposition
\[
 Dz=\dot z\,dt+D^sz
\]
and the singular part has a polar representation
$D^sz=\xi\,|D^sz|$ with $\xi(t)\in E_\sigma$ for
$|D^sz|$-almost every $t$.  The term $(z\cdot\nabla)\partial_tY$ belongs to $L^1(0,T;E)$ by the middle estimate in Lemma~\ref{lem:null} and Cauchy--Schwarz:
\[
 \|(z\cdot\nabla)\partial_tY\|_{L^1_TE}
 \lesssim
 \|z\|_{L^2_T\dot B^s_{p,1}}
 \|\partial_tY\|_{L^2_T\dot B^s_{p,1}}.
\]
The vector-valued BV product rule from
\cref{sec:bv-appendix}, applied to $L(t)z(t)=w(t)$, gives the $E$-valued measure identity
\begin{equation}\label{eq:measure-diff}
 (\partial_tw)\,dt=L(t)Dz+(z\cdot\nabla)\partial_tY\,dt.
\end{equation}
The rightmost term is therefore absolutely continuous.  Hence the singular part of \eqref{eq:measure-diff} is
\[
 L(t)D^sz=0.
\]
Equivalently, $L(t)\xi(t)=0$ for $|D^sz|$-almost every $t$.  Since
$\xi(t)\in E_\sigma$, the uniform coercivity of $L(t)$ implies
\[
 \|\xi(t)\|_E\le2\|L(t)\xi(t)\|_E=0.
\]
Thus $D^sz=0$ and $z\in W^{1,1}(0,T;E_\sigma)$.  Taking the absolutely continuous part of \eqref{eq:measure-diff} yields
\[
 \partial_tw=L_Y(\partial_tz)+(z\cdot\nabla)\partial_tY
\]
for almost every $t$.

Finally,
\[
 L_Y(\partial_tz)=\partial_tw-(z\cdot\nabla)\partial_tY.
\]
Coercivity and Lemma~\ref{lem:null}, now at middle regularity and with Hölder's inequality in time, give
\[
 \|\partial_tz\|_{L^1_TE}
 \lesssim
 \|\partial_tw\|_{L^1_TE}
 +\|z\|_{L^2_T\dot B^s_{p,1}}
  \|\partial_tY\|_{L^2_T\dot B^s_{p,1}},
\]
which is \eqref{eq:time-diff-est}.
\end{proof}

\begin{remark}\label{rem:solenoidal-singular-part}
The preceding proof is the only place where the domain of the graph must be the solenoidal subspace rather than the whole negative Besov space.  In measure form one has
$D(L_Yz)=L_YDz+(z\cdot\nabla)\partial_tY\,dt$.  The singular part is therefore $L_YD^sz$.  Since the polar vector of $D^sz$ remains solenoidal, coercivity of $L_Y$ on $E_\sigma$ forces $D^sz=0$.  Without this domain information the BV argument would not exclude a singular time derivative.
\end{remark}

\begin{lemma}[Product rule for compensated transport products]\label{lem:compensated-product-rule}
Let
\[
 a\in W^{1,1}(0,T;\dot B^{s-1}_{p,1})\cap L^2_T\dot B^s_{p,1},
 \qquad \operatorname{div}a=0,
\]
with $\partial_ta$ solenoidal in the distributional sense, and let
\[
 Y\in W^{1,1}(0,T;\dot B^{s+1}_{p,1}),
 \qquad \partial_tY\in L^2_T\dot B^s_{p,1}.
\]
Then
\[
 (a\cdot\nabla)Y\in W^{1,1}(0,T;\dot B^{s-1}_{p,1})
\]
and, almost everywhere in time,
\begin{equation}\label{eq:compensated-product-rule}
 \partial_t\bigl[(a\cdot\nabla)Y\bigr]
 = (\partial_ta\cdot\nabla)Y+(a\cdot\nabla)\partial_tY
 \quad\hbox{in }\dot B^{s-1}_{p,1}.
\end{equation}
Moreover,
\begin{align}\label{eq:compensated-product-rule-est}
 &\left\|\partial_t\bigl[(a\cdot\nabla)Y\bigr]\right\|_{L^1_T\dot B^{s-1}_{p,1}}\notag\\
 &\qquad\lesssim
 \|\partial_ta\|_{L^1_T\dot B^{s-1}_{p,1}}
 \|Y\|_{L^\infty_T\dot B^{s+1}_{p,1}}
 +\|a\|_{L^2_T\dot B^s_{p,1}}
  \|\partial_tY\|_{L^2_T\dot B^s_{p,1}}.
\end{align}
\end{lemma}

\begin{proof}
Extend $a$ and $Y$ slightly outside $(0,T)$ and apply standard space--time mollifiers.  Applying the Leray projector in the spatial variables to the approximants of $a$, we may choose smooth $a^m,Y^m$ such that $a^m$ and $\partial_ta^m$ are solenoidal and
\[
 a^m\to a\quad\hbox{in }W^{1,1}_T\dot B^{s-1}_{p,1}
       \cap L^2_T\dot B^s_{p,1},
\]
while
\[
 Y^m\to Y\quad\hbox{in }W^{1,1}_T\dot B^{s+1}_{p,1},
 \qquad
 \partial_tY^m\to\partial_tY
 \quad\hbox{in }L^2_T\dot B^s_{p,1}.
\]
For smooth functions the identity is the ordinary product rule.  The
term $(\partial_ta^m\cdot\nabla)Y^m$ converges in
$L^1_T\dot B^{s-1}_{p,1}$ by Lemma~\ref{lem:null} in the low estimate,
whereas $(a^m\cdot\nabla)\partial_tY^m$ converges by the middle estimate
\eqref{eq:null-time-2}.  The same estimates show that
$(a^m\cdot\nabla)Y^m\to(a\cdot\nabla)Y$ in
$L^\infty_T\dot B^{s-1}_{p,1}$.  Passing to the limit gives
\eqref{eq:compensated-product-rule}; the bound
\eqref{eq:compensated-product-rule-est} follows directly from the two
estimates just used.
\end{proof}

\section{General Lagrangian formulation and short-time bounds}\label{sec:lagrangian}

Let $(u_i,r_i,P_i)$, $i=1,2$, be two solutions on a common interval.  Let $X_i$ be the volume-preserving flow of $u_i$ and set
\[
 Y_i=X_i-\Id,\qquad F_i=\D X_i=I+\D Y_i,
 \qquad A_i=F_i^{-1},
\]
\[
 v_i=u_i\circ X_i,\qquad R_i=r_i\circ X_i,
 \qquad q_i=P_i\circ X_i.
\]
Here $F_i=I+\D Y_i$ is the distributional derivative constructed in
Corollary~\ref{cor:flow-piola}, and the affine Jacobian facts are those of Proposition~\ref{prop:critical-flow-jacobian}.  Thus
$F_i-I\in\B^s_{p,1}\hookrightarrow L^\infty$, $\det F_i=1$, and the inverse $A_i=F_i^{-1}$ is understood in the affine
Besov algebra, equivalently by the cofactor formula
$A_i=(\operatorname{cof}F_i)^\top$.
Our convention is
$(F_i)_{k\ell}=\partial_\ell(X_i)_k$ and
$(\nabla v_i)_{k\ell}=\partial_\ell(v_i)_k$.
Since $\det F_i=1$, the Piola identity reads
\begin{equation}\label{eq:piola-general}
 \partial_\alpha(A_i)_{\alpha\beta}=0.
\end{equation}

We record the transformed equations in detail.  Along the flow,
\[
 \partial_tv_i=(\partial_tu_i+u_i\cdot\nabla_xu_i)\circ X_i.
\]
The chain rule gives
\[
 \nabla_yv_i=((\nabla_xu_i)\circ X_i)F_i,
 \qquad
 (\nabla_xu_i)\circ X_i=\nabla_yv_iA_i.
\]
Thus the pulled-back velocity gradient is
\begin{equation}\label{eq:M-def-gradient}
 M_i:=(\nabla_xu_i)\circ X_i=\nabla_yv_iA_i.
\end{equation}
Since the equilibrium $\bar a$ is constant, composition of the internal-variable equation with $X_i$ removes the transport term and gives the Banach-space ODE
\begin{equation}\label{eq:internal-ode}
 \partial_tR_i=\mathscr G(R_i,M_i),
 \qquad R_i(0)=r_{0,i}.
\end{equation}

For the pressure,
\[
 \nabla_yq_i=F_i^\top((\nabla_xP_i)\circ X_i),
 \qquad
 (\nabla_xP_i)\circ X_i=A_i^\top\nabla_yq_i.
\]
For the viscous term, componentwise,
\begin{align*}
 ((\Delta_xu_i)\circ X_i)_\beta
 &=A_{i,\alpha k}\partial_\alpha
    \bigl(A_{i,\gamma k}\partial_\gamma(v_i)_\beta\bigr)\\
 &=\partial_\alpha\bigl((A_iA_i^\top)_{\alpha\gamma}
       \partial_\gamma(v_i)_\beta\bigr).
\end{align*}
In the last equality we used the Piola identity.  With $C_i=A_iA_i^\top$, this is the row-wise divergence of $\nabla_yv_i\,C_i$, because
\[
 (\nabla_yv_i\,C_i)_{\beta\alpha}
 =\partial_\gamma(v_i)_\beta(C_i)_{\gamma\alpha}
 =\partial_\gamma(v_i)_\beta(C_i)_{\alpha\gamma}.
\]
Thus the matrix order in the Lagrangian viscous tensor is $\nabla v_i\,C_i$.  More generally, if the row-wise divergence of a matrix $S$ is
$(\diver_x S)_\beta=\partial_{x_k}S_{\beta k}$, then
\begin{align}
 ((\diver_xS)\circ X_i)_\beta
 &=A_{i,\alpha k}\partial_\alpha(S_{\beta k}\circ X_i)\notag\\
 &=\partial_\alpha\bigl((S\circ X_i)_{\beta k}
 A_{i,\alpha k}\bigr),
 \label{eq:stress-change}
\end{align}
that is,
\[
 (\diver_x S)\circ X_i
 =\diver_y\bigl((S\circ X_i)A_i^\top\bigr).
\]
For smooth fields, the pulled-back momentum equation is
\begin{equation}\label{eq:lag-mom-general}
 \partial_tv_i-\nu\operatorname{div}(\nabla v_i\,A_iA_i^\top)
 +A_i^\top\nabla q_i
 =\operatorname{div}\bigl(\mathscr S(R_i)A_i^\top\bigr).
\end{equation}
At the critical regularity this display is not used as a definition of
$A_i^\top\nabla q_i$.  The product
$(A_i-I)^\top\nabla q_i$ would be exactly of the forbidden type
$\dot B^s_{p,1}\cdot\dot B^{s-1}_{p,1}$.  Instead, using the Piola
identity, we always rewrite the pressure contribution in divergence
form:
\begin{equation}\label{eq:pressure-piola-single}
 (I-A_i^\top)\nabla q_i
 =\operatorname{div}\bigl((I-A_i^\top)q_i\bigr),
\end{equation}
where the divergence is row-wise, i.e.
$\bigl(\operatorname{div}M\bigr)_\beta=\partial_\alpha M_{\beta\alpha}$.
Thus the critical Lagrangian momentum equation is the perturbative
Stokes identity
\begin{equation}\label{eq:lag-mom-critical}
 \partial_tv_i-\nu\Delta v_i+\nabla q_i
 =\operatorname{div}\Bigl(
 \nu\nabla v_i\,(A_iA_i^\top-I)
 +(I-A_i^\top)q_i
 +\mathscr S(R_i)A_i^\top\Bigr).
\end{equation}
Here and below all affine products are expanded before estimating:
\[
 A_iA_i^\top-I=(A_i-I)+(A_i-I)^\top
 +(A_i-I)(A_i-I)^\top,
\]
\[
 \mathscr S(R_i)A_i^\top
 =\mathscr S(R_i)+\mathscr S(R_i)(A_i-I)^\top,
 \qquad
 (I-A_i^\top)q_i=-(A_i-I)^\top q_i.
\]
Consequently only genuine $\dot B^s_{p,1}$ quantities are multiplied.  In particular, the viscous perturbation is the usual matrix product $\nabla v_i\,(A_iA_i^\top-I)$ followed by row-wise divergence.
The classical form \eqref{eq:lag-mom-general} is retained only as a
mnemonic for smooth approximations; the rigorous critical equation is
\eqref{eq:lag-mom-critical}.
Finally,
\[
 (A_i)_{\alpha\beta}\partial_\alpha(v_i)_\beta
 =((\partial_{x_\beta}u_{i,\beta})\circ X_i)=0,
\]
so the incompressibility constraint is
\begin{equation}\label{eq:lag-div-general}
 \mathcal D_{A_i}v_i:=(A_i)_{\alpha\beta}
 \partial_\alpha(v_i)_\beta=0.
\end{equation}
For smooth fields this is equivalent to $\diver(A_iv_i)=0$.  At critical regularity the latter product is not used; instead
\[
 \mathcal D_{A_i}v_i=\diver v_i+
 \sum_{\alpha,\beta}(A_i-I)_{\alpha\beta}\partial_\alpha(v_i)_\beta,
\]
and the last term belongs to $L^1_T\B^s_{p,1}$ by the algebra property.  All identities are first obtained for smooth approximations and then passed to the critical class using Lemma~\ref{lem:composition} and Proposition~\ref{prop:critical-flow-jacobian}.

\begin{lemma}[Lagrangian regularity and short-time smallness]\label{lem:lag-regularity}
For each $i$,
\begin{align}
 v_i&\in C_T\B^{s-1}_{p,1}
 \cap L^1_T\B^{s+1}_{p,1}
 \cap L^2_T\B^s_{p,1},\label{eq:v-reg}\\
 q_i&\in L^1_T\B^s_{p,1},
 \qquad \partial_tv_i\in L^1_T\B^{s-1}_{p,1},\label{eq:v-time-reg}\\
 R_i&\in C_T\B^s_{p,1}(\mathcal W),
 \qquad \partial_tR_i\in L^1_T\B^s_{p,1}.\label{eq:R-reg}
\end{align}
Let
\begin{align}
 \eta_T:=\sum_{i=1}^2\Bigl(&
 \norm{v_i}{L^2_T\B^s_{p,1}}
 +\norm{v_i}{L^1_T\B^{s+1}_{p,1}}
 +\norm{\partial_tv_i}{L^1_T\B^{s-1}_{p,1}}
 +\norm{q_i}{L^1_T\B^s_{p,1}}\Bigr)+T.
 \label{eq:eta-general}
\end{align}
For a fixed pair of solutions, $\eta_T\to0$ as $T\downarrow0$.  For data converging to fixed data this convergence is uniform on the common interval of Proposition~\ref{prop:local-input}.  Moreover, for all sufficiently small $T$,
\begin{align}
 &\norm{Y_i}{L^\infty_T\B^{s+1}_{p,1}}
 +\norm{F_i-I}{L^\infty_T\B^s_{p,1}}
 +\norm{A_i-I}{L^\infty_T\B^s_{p,1}}
 \lesssim\eta_T,\label{eq:flow-small}\\
 &\norm{Y_i}{L^\infty_T\B^s_{p,1}}
 \lesssim T^{1/2}\norm{v_i}{L^2_T\B^s_{p,1}}.
 \label{eq:flow-displacement}
\end{align}
\end{lemma}

\begin{proof}
We suppress the index $i$.  Since $s-1\in(-1,0)$ and $X(t)$ is volume preserving with uniform bi-Lipschitz bounds on bounded time intervals, Lemma~\ref{lem:composition} gives
\[
 \norm{v(t)}{\B^{s-1}_{p,1}}
 \le C_X\norm{u(t)}{\B^{s-1}_{p,1}}.
\]
Strong time continuity follows from the time-dependent part of the same lemma.

Write $G=F-I=\D Y$.  By the smooth-flow approximation construction in Proposition~\ref{prop:critical-flow-jacobian}, $G$ satisfies the integral identity
\[
 G(t)=\int_0^t ((\nabla u)\circ X)(\tau)(I+G(\tau))\,\dd\tau .
\]
Equivalently, in $L^1(0,T;\B^s_{p,1})$,
\[
 \partial_tG=(\nabla u)\circ X+((\nabla u)\circ X)G,
 \qquad G(0)=0.
\]
Composition at index $s$ and the algebra property imply
\[
 \norm{G(t)}{\B^s_{p,1}}
 \le C_X\int_0^t\norm{u(\tau)}{\B^{s+1}_{p,1}}
 \bigl(1+\norm{G(\tau)}{\B^s_{p,1}}\bigr)\dd\tau.
\]
Gronwall's lemma yields $G\in C_T\B^s_{p,1}$ and
\begin{equation}\label{eq:G-flow-bound}
 \norm{G}{L^\infty_T\B^s_{p,1}}
 \le C_X\exp\bigl(C_X\norm{u}{L^1_T\B^{s+1}_{p,1}}\bigr)
 \norm{u}{L^1_T\B^{s+1}_{p,1}}.
\end{equation}
At this point the cofactor formula gives $A=(\operatorname{cof}F)^\top$.
Since $\det F=1$, $A-I$ is a finite polynomial in $G=F-I$ without
constant term; hence $A-I\in C_T\B^s_{p,1}$ by the algebra property.
Thus multiplication by $A=I+(A-I)$ preserves the $\B^s_{p,1}$ estimates
used below.

The chain rule reads
\[
 \nabla v=((\nabla u)\circ X)(I+G).
\]
Thus
\[
 \norm{\nabla v(t)}{\B^s_{p,1}}
 \le C_X\norm{u(t)}{\B^{s+1}_{p,1}}
 \bigl(1+\norm{G(t)}{\B^s_{p,1}}\bigr),
\]
and integration gives $v\in L^1_T\B^{s+1}_{p,1}$.  Lemma~\ref{lem:standard-besov}(iii) then gives
$v\in L^2_T\B^s_{p,1}$.

For the pressure, $q=P\circ X$ and composition at index $s$ yields
$q\in L^1_T\B^s_{p,1}$.  For the time derivative it is essential to use the material derivative:
\[
 \partial_tv=(\partial_tu+u\cdot\nabla u)\circ X=(D_tu)\circ X.
\]
Proposition~\ref{prop:local-input} and composition at index $s-1$ give
$\partial_tv\in L^1_T\B^{s-1}_{p,1}$.

Next, \eqref{eq:M-def-gradient}, the algebra property, and the bound for $A$ imply
\[
 \norm{M}{L^1_T\B^s_{p,1}}
 \le C_X\norm{v}{L^1_T\B^{s+1}_{p,1}}.
\]
Since $R\in C_T\B^s_{p,1}$ by composition, the ODE
\eqref{eq:internal-ode} and Lemma~\ref{lem:constitutive-tame} show that
$\partial_tR\in L^1_T\B^s_{p,1}$.

The time-integrable Eulerian norms in Proposition~\ref{prop:local-input} are absolutely continuous at zero.  The composition constants are uniform whenever
$\int_0^T\norm{\nabla u}{L^\infty}\dd t$ is uniformly bounded.  Hence
\[
 \norm{v}{L^1_T\B^{s+1}_{p,1}},\quad
 \norm{q}{L^1_T\B^s_{p,1}},\quad
 \norm{\partial_tv}{L^1_T\B^{s-1}_{p,1}}
 \longrightarrow0
\]
as $T\downarrow0$, uniformly for convergent data.  Interpolation gives the same conclusion for the $L^2_T\B^s_{p,1}$ norm, proving $\eta_T\to0$.

Finally,
\[
 Y(t)=\int_0^t v(\tau)\dd\tau.
\]
Therefore
\[
 \norm{Y}{L^\infty_T\B^{s+1}_{p,1}}
 \le\norm{v}{L^1_T\B^{s+1}_{p,1}},
 \qquad
 \norm{Y}{L^\infty_T\B^s_{p,1}}
 \le T^{1/2}\norm{v}{L^2_T\B^s_{p,1}}.
\]
The derivative equivalence controls $F-I=\D Y$.  To get the short-time
linear bound for the inverse, put $K=A-I$.  Since $(I+K)(I+G)=I$, we have
$K=-G-KG$.  For short time the last term is absorbed in the algebra
$\B^s_{p,1}$, giving the asserted estimate for $A-I$.
\end{proof}

\begin{proposition}[Stability of the Lagrangian internal variable]\label{prop:internal-stability}
For any quantity $f$, write $\delta f=f_1-f_2$, and define
\begin{equation}\label{eq:Z-general}
 \cZ_T:=
 \norm{\delta v}{L^\infty_T\B^{s-1}_{p,1}}
 +\norm{\partial_t\delta v}{L^1_T\B^{s-1}_{p,1}}
 +\norm{\nabla^2\delta v}{L^1_T\B^{s-1}_{p,1}}
 +\norm{\nabla\delta q}{L^1_T\B^{s-1}_{p,1}}.
\end{equation}
On a common interval on which the solution norms are bounded by a fixed constant $\cM$, one has
\begin{equation}\label{eq:R-difference}
 \norm{\delta R}{L^\infty_T\B^s_{p,1}}
 \le C_{\cM}\left(
  \norm{\delta r_0}{\B^s_{p,1}}+\cZ_T\right).
\end{equation}
\end{proposition}

\begin{proof}
The derivative equivalence gives
\[
 \norm{\delta v}{L^1_T\B^{s+1}_{p,1}}
 \lesssim\norm{\nabla^2\delta v}{L^1_T\B^{s-1}_{p,1}},
 \qquad
 \norm{\delta q}{L^1_T\B^s_{p,1}}
 \lesssim\norm{\nabla\delta q}{L^1_T\B^{s-1}_{p,1}}.
\]
Interpolating the low and high velocity norms yields
\begin{equation}\label{eq:interp-dv-general}
 \norm{\delta v}{L^2_T\B^s_{p,1}}
 +\norm{\delta v}{L^1_T\B^{s+1}_{p,1}}
 +\norm{\delta q}{L^1_T\B^s_{p,1}}
 \le C\cZ_T.
\end{equation}
Since $\delta Y(t)=\int_0^t\delta v(\tau)\dd\tau$,
\[
 \norm{\delta Y}{L^\infty_T\B^{s+1}_{p,1}}
 \le\norm{\delta v}{L^1_T\B^{s+1}_{p,1}},
 \qquad
 \norm{\delta Y}{L^\infty_T\B^s_{p,1}}
 \le T^{1/2}\norm{\delta v}{L^2_T\B^s_{p,1}}.
\]
Thus $\delta F=\D\delta Y$ is controlled in $L^\infty_T\B^s_{p,1}$.  The inverse-matrix identity
\[
 A_1-A_2=A_1(F_2-F_1)A_2
\]
and the algebra property give
\begin{align}
 &\norm{\delta Y}{L^\infty_T\B^{s+1}_{p,1}}
 +\norm{\delta F}{L^\infty_T\B^s_{p,1}}
 +\norm{\delta A}{L^\infty_T\B^s_{p,1}}
 \le C_{\cM}\cZ_T,\label{eq:dY-high-general}\\
 &\norm{\delta Y}{L^\infty_T\B^s_{p,1}}
 \le C T^{1/2}\cZ_T.\label{eq:dY-low-general}
\end{align}

Next,
\[
 \delta M=(\nabla\delta v)A_1+(\nabla v_2)\delta A.
\]
Hence
\begin{align*}
 \norm{\delta M}{L^1_T\B^s_{p,1}}
 &\le C_{\cM}\norm{\delta v}{L^1_T\B^{s+1}_{p,1}}
 +C\norm{\nabla v_2}{L^1_T\B^s_{p,1}}
       \norm{\delta A}{L^\infty_T\B^s_{p,1}}\\
 &\le C_{\cM}\cZ_T.
\end{align*}
Subtracting \eqref{eq:internal-ode} and using \eqref{eq:G-tame}, we obtain for almost every $t$
\begin{align*}
 \norm{\partial_t\delta R(t)}{\B^s_{p,1}}
 \le C_{\cM}\Bigl[&
 (1+\norm{M_1(t)}{\B^s_{p,1}}+
       \norm{M_2(t)}{\B^s_{p,1}})
       \norm{\delta R(t)}{\B^s_{p,1}}\\
 &+\norm{\delta M(t)}{\B^s_{p,1}}\Bigr].
\end{align*}
Since $R_i$ is absolutely continuous as a $\B^s_{p,1}$-valued
function, the preceding differential inequality implies, for every
$t\in[0,T]$,
\begin{align*}
 \norm{\delta R(t)}{\B^s_{p,1}}
 &\le \norm{\delta r_0}{\B^s_{p,1}}
 +C_{\cM}\int_0^t
 (1+\norm{M_1(\tau)}{\B^s_{p,1}}
     +\norm{M_2(\tau)}{\B^s_{p,1}})
     \norm{\delta R(\tau)}{\B^s_{p,1}}\,\dd\tau\\
 &\qquad
 +C_{\cM}\int_0^t\norm{\delta M(\tau)}{\B^s_{p,1}}\,\dd\tau .
\end{align*}
Since $M_i\in L^1_T\B^s_{p,1}$, Gronwall's lemma gives
\[
 \norm{\delta R}{L^\infty_T\B^s_{p,1}}
 \le C_{\cM}\left(
 \norm{\delta r_0}{\B^s_{p,1}}
 +\norm{\delta M}{L^1_T\B^s_{p,1}}\right),
\]
and the preceding bound for $\delta M$ yields \eqref{eq:R-difference}.
\end{proof}

\section{Piola difference estimates and abstract stability}\label{sec:stability}

Throughout this section, $\delta f=f_1-f_2$.  To avoid ambiguity with the row-wise divergence convention, we write
\[
 B\diamond\nabla w:=\sum_{\alpha,\beta}B_{\alpha\beta}\partial_\alpha w_\beta
\]
for the Lagrangian divergence contraction of a matrix field $B$ with the gradient of a vector field $w$.

Take $T$ so small that
\begin{equation}\label{eq:small-choice}
 \max_i\norm{Y_i}{L^\infty_T\B^{s+1}_{p,1}}
 \le\frac{\varepsilon_0}{2}.
\end{equation}
Corollary~\ref{cor:flow-piola} gives unique fields $z_i$ such that
\begin{equation}\label{eq:z-def-rigorous}
 \diver z_i=0,
 \qquad v_i=z_i+(z_i\cdot\nabla)Y_i.
\end{equation}
They satisfy
\begin{align}
 \norm{z_i}{L^\infty_T\B^{s-1}_{p,1}}&\le C_{\cM},\label{eq:z-low}\\
 \norm{z_i}{L^2_T\B^s_{p,1}}&\lesssim\eta_T,\label{eq:z-mid}\\
 \norm{\partial_tz_i}{L^1_T\B^{s-1}_{p,1}}&\lesssim\eta_T.\label{eq:zt-est}
\end{align}
Indeed, the first two bounds follow from the closed Piola graph.  Lemma~\ref{lem:time-diff} applied to \eqref{eq:z-def-rigorous}, with $\partial_tY_i=v_i$, gives more precisely
\[
 \|\partial_tz_i\|_{L^1_T\dot B^{s-1}_{p,1}}
 \lesssim \|\partial_tv_i\|_{L^1_T\dot B^{s-1}_{p,1}}
 +\|z_i\|_{L^2_T\dot B^s_{p,1}}\|v_i\|_{L^2_T\dot B^s_{p,1}}
 \lesssim \eta_T+\eta_T^2,
\]
and this is bounded by $C\eta_T$ after reducing $T$ so that $\eta_T\le1$.

\begin{proposition}[Difference estimates for the Piola velocities]\label{prop:diff-piola}
The identities
\begin{equation}\label{eq:dz-identity}
 \delta v=\delta z+(\delta z\cdot\nabla)Y_1
 +(z_2\cdot\nabla)\delta Y
\end{equation}
and
\begin{align}
 \partial_t\delta v={}&\partial_t\delta z
 +(\partial_t\delta z\cdot\nabla)Y_1
 +(\partial_tz_2\cdot\nabla)\delta Y\notag\\
 &+(\delta z\cdot\nabla)v_1
 +(z_2\cdot\nabla)\delta v
 \label{eq:dzt-identity}
\end{align}
hold in the critical spaces.  Moreover,
\begin{align}
 \norm{\delta z}{L^\infty_T\B^{s-1}_{p,1}}
 +\norm{\delta z}{L^2_T\B^s_{p,1}}
 &\le C_{\cM}\cZ_T,\label{eq:dz-est}\\
 \norm{\partial_t\delta z}{L^1_T\B^{s-1}_{p,1}}
 &\le C_{\cM}\cZ_T.\label{eq:dzt-est}
\end{align}
\end{proposition}

\begin{proof}
Subtract \eqref{eq:z-def-rigorous} to obtain \eqref{eq:dz-identity}.  Put
$H=(z_2\cdot\nabla)\delta Y$.  Since
$\delta v-H=L_{Y_1}\delta z$, coercivity gives the low and middle bounds for $\delta z$ in terms of $\delta v$ and $H$.  At low regularity, Lemma~\ref{lem:null} gives
\[
 \|H\|_{L^\infty_T\dot B^{s-1}_{p,1}}
 \lesssim
 \|z_2\|_{L^\infty_T\dot B^{s-1}_{p,1}}
 \|\delta Y\|_{L^\infty_T\dot B^{s+1}_{p,1}}.
\]
For the $L^2_T\dot B^s_{p,1}$ estimate one uses the algebra property, rather than the low-output compensated estimate:
\[
 \|H\|_{L^2_T\dot B^s_{p,1}}
 \lesssim
 \|z_2\|_{L^2_T\dot B^s_{p,1}}
 \|\delta Y\|_{L^\infty_T\dot B^{s+1}_{p,1}}.
\]
Together with \eqref{eq:z-low}--\eqref{eq:z-mid} and \eqref{eq:dY-high-general}, these estimates yield \eqref{eq:dz-est}.  The hypotheses of
Lemma~\ref{lem:compensated-product-rule} are satisfied with
$a=z_2$ and $Y=\delta Y$, because
$z_2\in W^{1,1}_T\dot B^{s-1}_{p,1}\cap L^2_T\dot B^s_{p,1}$,
$\delta Y\in W^{1,1}_T\dot B^{s+1}_{p,1}$, and
$\partial_t\delta Y=\delta v\in L^2_T\dot B^s_{p,1}$.  Hence
$H\in W^{1,1}_T\dot B^{s-1}_{p,1}$ and
\[
 \partial_tH=(\partial_tz_2\cdot\nabla)\delta Y
 +(z_2\cdot\nabla)\delta v.
\]
The low and middle compensated estimates yield
\[
 \norm{\partial_tH}{L^1_T\B^{s-1}_{p,1}}
 \le C_{\cM}\eta_T\cZ_T.
\]
Since $\delta v-H=L_{Y_1}\delta z$, Lemma~\ref{lem:time-diff} gives \eqref{eq:dzt-est}.  Formula \eqref{eq:dzt-identity} is the resulting almost-everywhere derivative of \eqref{eq:dz-identity}.
\end{proof}

\begin{proposition}[Stokes estimate with nonzero divergence]\label{prop:stokes}
Let $1<p<\infty$, $\rho\in\R$, and $\nu>0$.  Suppose
\begin{equation}\label{eq:abstract-stokes}
 \partial_tu-\nu\Delta u+\nabla\pi=\diver F,
 \qquad \diver u=g,
 \qquad u|_{t=0}=u_0,
\end{equation}
where $\diver F$ is understood row-wise and
\[
 F\in L^1_T\B^{\rho+1}_{p,1},\qquad
 \mathcal R,\nabla g\in L^1_T\B^\rho_{p,1},\qquad
 \partial_tg=\diver \mathcal R,
 \qquad g(0)=\diver u_0.
\]
Then
\begin{align}
 &\norm{u}{L^\infty_T\B^\rho_{p,1}}
 +\norm{\partial_tu}{L^1_T\B^\rho_{p,1}}
 +\nu\norm{\nabla^2u}{L^1_T\B^\rho_{p,1}}
 +\norm{\nabla\pi}{L^1_T\B^\rho_{p,1}}\notag\\
 &\qquad\le C_{p,d,\nu}\left(
 \norm{u_0}{\B^\rho_{p,1}}
 +\norm{F}{L^1_T\B^{\rho+1}_{p,1}}
 +\norm{\mathcal R}{L^1_T\B^\rho_{p,1}}
 +\norm{\nabla g}{L^1_T\B^\rho_{p,1}}\right).
 \label{eq:stokes-est}
\end{align}
\end{proposition}

\begin{proof}
Let $\PP$ be the Leray projector and $\QQ=I-\PP$.  Applying $\PP$ to \eqref{eq:abstract-stokes} gives
\[
 \partial_t\PP u-\nu\Delta\PP u=\PP\diver F,
 \qquad \PP u|_{t=0}=\PP u_0.
\]
The homogeneous heat maximal-regularity estimate yields
\begin{align*}
 &\|\PP u\|_{L^\infty_T\dot B^\rho_{p,1}}
 +\|\partial_t\PP u\|_{L^1_T\dot B^\rho_{p,1}}
 +\nu\|\nabla^2\PP u\|_{L^1_T\dot B^\rho_{p,1}}\\
 &\qquad\lesssim_{p,d,\nu}
 \|u_0\|_{\dot B^\rho_{p,1}}
 +\|F\|_{L^1_T\dot B^{\rho+1}_{p,1}}.
\end{align*}

The potential part is determined by the divergence:
\[
 \QQ u=\nabla\Delta^{-1}g.
\]
Since $g(0)=\diver u_0$ and
$g(t)=g(0)+\int_0^t\diver\mathcal R(\tau)\,d\tau$, Fourier multiplier estimates give
\[
 \|\QQ u\|_{L^\infty_T\dot B^\rho_{p,1}}
 \lesssim
 \|u_0\|_{\dot B^\rho_{p,1}}
 +\|\mathcal R\|_{L^1_T\dot B^\rho_{p,1}}.
\]
Moreover,
\[
 \partial_t\QQ u
 =\nabla\Delta^{-1}\diver\mathcal R,
\]
so
\[
 \|\partial_t\QQ u\|_{L^1_T\dot B^\rho_{p,1}}
 \lesssim\|\mathcal R\|_{L^1_T\dot B^\rho_{p,1}}.
\]
Finally, $\nabla^2\QQ u$ is a Fourier multiplier of order one applied to $g$; hence
\[
 \|\nabla^2\QQ u\|_{L^1_T\dot B^\rho_{p,1}}
 \lesssim\|\nabla g\|_{L^1_T\dot B^\rho_{p,1}}.
\]

Applying $\QQ$ to the momentum equation and using
$\QQ\nabla\pi=\nabla\pi$ gives
\[
 \nabla\pi
 =\QQ\diver F-\partial_t\QQ u+\nu\Delta\QQ u.
\]
The first term is controlled by
$\|F\|_{\dot B^{\rho+1}_{p,1}}$, and the other two terms have just been estimated.  Adding the solenoidal and potential bounds proves \eqref{eq:stokes-est}.  The compatibility condition at $t=0$ is exactly what identifies the initial potential part with $\nabla\Delta^{-1}g(0)$.
\end{proof}

\begin{proposition}[Control of the nonzero divergence]\label{prop:nonzero-divergence}
Set $g=\diver\delta v$.  Then
\begin{equation}\label{eq:grad-g}
 \norm{\nabla g}{L^1_T\B^{s-1}_{p,1}}
 \le C_{\cM}\eta_T\cZ_T.
\end{equation}
Moreover, $\partial_tg=\diver R_g$, where
\begin{align}
 R_g={}&(\partial_t\delta z\cdot\nabla)Y_1
 +(\delta z\cdot\nabla)v_1
 +(\partial_tz_2\cdot\nabla)\delta Y
 +(z_2\cdot\nabla)\delta v,
 \label{eq:Rg}
\end{align}
and
\begin{equation}\label{eq:Rg-est}
 \norm{R_g}{L^1_T\B^{s-1}_{p,1}}
 \le C_{\cM}\eta_T\cZ_T.
\end{equation}
\end{proposition}

\begin{proof}
From \eqref{eq:lag-div-general},
\begin{equation}\label{eq:g-algebra-general}
 g=(I-A_1)\diamond\nabla\delta v+(A_2-A_1)\diamond\nabla v_2,
\end{equation}
or explicitly
\[
 g=\sum_{\alpha,\beta}(\delta_{\alpha\beta}-A_{1,\alpha\beta})
     \partial_\alpha(\delta v)_\beta
   +\sum_{\alpha,\beta}(A_{2,\alpha\beta}-A_{1,\alpha\beta})
     \partial_\alpha(v_2)_\beta .
\]
The two products on the right belong to $L^1_T\B^s_{p,1}$.  Indeed,
\begin{align*}
 \norm{(I-A_1)\diamond\nabla\delta v}{L^1_T\B^s_{p,1}}
 &\le C\norm{I-A_1}{L^\infty_T\B^s_{p,1}}
       \norm{\delta v}{L^1_T\B^{s+1}_{p,1}},\\
 \norm{(A_2-A_1)\diamond\nabla v_2}{L^1_T\B^s_{p,1}}
 &\le C\norm{A_2-A_1}{L^\infty_T\B^s_{p,1}}
       \norm{v_2}{L^1_T\B^{s+1}_{p,1}}.
\end{align*}
Using \eqref{eq:flow-small}, \eqref{eq:dY-high-general}, and
\eqref{eq:interp-dv-general}, and then one derivative equivalence, gives
\eqref{eq:grad-g}.

Since $\diver\delta z=0$, \eqref{eq:dz-identity} yields
\[
 g=\diver G_g,
 \qquad
 G_g=(\delta z\cdot\nabla)Y_1+(z_2\cdot\nabla)\delta Y.
\]
Lemma~\ref{lem:compensated-product-rule}, applied to both terms in
$G_g$, gives $G_g\in W^{1,1}_T\dot B^{s-1}_{p,1}$ and
$\partial_tG_g=R_g$.  For the four terms of $R_g$, Lemma~\ref{lem:null} gives
\begin{align*}
 \norm{(\partial_t\delta z\cdot\nabla)Y_1}{L^1_T\B^{s-1}_{p,1}}
 &\lesssim
 \norm{\partial_t\delta z}{L^1_T\B^{s-1}_{p,1}}
 \norm{Y_1}{L^\infty_T\B^{s+1}_{p,1}},\\
 \norm{(\delta z\cdot\nabla)v_1}{L^1_T\B^{s-1}_{p,1}}
 &\lesssim
 \norm{\delta z}{L^\infty_T\B^{s-1}_{p,1}}
 \norm{v_1}{L^1_T\B^{s+1}_{p,1}},\\
 \norm{(\partial_tz_2\cdot\nabla)\delta Y}{L^1_T\B^{s-1}_{p,1}}
 &\lesssim
 \norm{\partial_tz_2}{L^1_T\B^{s-1}_{p,1}}
 \norm{\delta Y}{L^\infty_T\B^{s+1}_{p,1}},\\
 \norm{(z_2\cdot\nabla)\delta v}{L^1_T\B^{s-1}_{p,1}}
 &\lesssim
 \norm{z_2}{L^2_T\B^s_{p,1}}
 \norm{\delta v}{L^2_T\B^s_{p,1}}.
\end{align*}
Estimates \eqref{eq:z-mid}, \eqref{eq:zt-est}, \eqref{eq:dz-est},
\eqref{eq:dzt-est}, \eqref{eq:dY-high-general}, and
\eqref{eq:interp-dv-general} give \eqref{eq:Rg-est}.
\end{proof}

\subsection{Difference momentum equation}
Put $C_i=A_iA_i^\top$.  In accordance with \eqref{eq:lag-mom-critical}, the viscous perturbation is written as $\nabla v_i(C_i-I)$ before applying row-wise divergence.  The Piola identity gives
\begin{equation}\label{eq:pressure-piola-general}
 (I-A_i^\top)\nabla q_i
 =\diver\bigl((I-A_i^\top)q_i\bigr).
\end{equation}
Subtracting the rigorous perturbative form \eqref{eq:lag-mom-critical} gives
\begin{equation}\label{eq:stokes-diff-general}
 \partial_t\delta v-\nu\Delta\delta v+\nabla\delta q
 =\diver\cF,
 \qquad \diver\delta v=g,
 \qquad \delta v|_{t=0}=\delta u_0,
\end{equation}
where
\begin{align}
 \cF={}&\nu\nabla\delta v\,(C_1-I)
 +\nu\nabla v_2\,(C_1-C_2)\notag\\
 &+(I-A_1^\top)\delta q+(A_2^\top-A_1^\top)q_2\notag\\
 &+\mathscr S(R_1)A_1^\top-\mathscr S(R_2)A_2^\top.
 \label{eq:Fcal-general}
\end{align}
We estimate every term at the stress regularity $\B^s_{p,1}$.  Since $C_i=A_iA_i^\top$ and $\B^s_{p,1}$ is an algebra,
\[
 \norm{C_1-I}{L^\infty_T\B^s_{p,1}}
 \le C_{\cM}\eta_T,
 \qquad
 \norm{C_1-C_2}{L^\infty_T\B^s_{p,1}}
 \le C_{\cM}\cZ_T.
\]
Consequently,
\begin{align*}
 \norm{\nabla\delta v\,(C_1-I)}{L^1_T\B^s_{p,1}}
 &\le C_{\cM}\eta_T
       \norm{\delta v}{L^1_T\B^{s+1}_{p,1}}
 \le C_{\cM}\eta_T\cZ_T,\\
 \norm{\nabla v_2\,(C_1-C_2)}{L^1_T\B^s_{p,1}}
 &\le C_{\cM}\cZ_T
       \norm{v_2}{L^1_T\B^{s+1}_{p,1}}
 \le C_{\cM}\eta_T\cZ_T.
\end{align*}
The pressure terms obey the same pattern:
\begin{align*}
 \norm{(I-A_1^\top)\delta q}{L^1_T\B^s_{p,1}}
 &\le C_{\cM}\eta_T
       \norm{\delta q}{L^1_T\B^s_{p,1}}
 \le C_{\cM}\eta_T\cZ_T,\\
 \norm{(A_2^\top-A_1^\top)q_2}{L^1_T\B^s_{p,1}}
 &\le C_{\cM}\cZ_T
       \norm{q_2}{L^1_T\B^s_{p,1}}
 \le C_{\cM}\eta_T\cZ_T.
\end{align*}
Thus
\begin{equation}\label{eq:Fcal-geom}
 \norm{\cF_{\rm geom}}{L^1_T\B^s_{p,1}}
 \le C_{\cM}\eta_T\cZ_T.
\end{equation}

For the constitutive stress, split in the affine Besov algebra:
\begin{align*}
 \mathscr S(R_1)A_1^\top-\mathscr S(R_2)A_2^\top
 ={}&\bigl(\mathscr S(R_1)-\mathscr S(R_2)\bigr)
   +\bigl(\mathscr S(R_1)-\mathscr S(R_2)\bigr)(A_1-I)^\top\\
 &+\mathscr S(R_2)(A_1^\top-A_2^\top).
\end{align*}
Thus no homogeneous Besov norm is assigned to the constant identity matrix.  The solution bounds, \eqref{eq:stress-tame}, Proposition~\ref{prop:internal-stability}, and \eqref{eq:dY-high-general} imply
\begin{align}
 &\norm{\mathscr S(R_1)A_1^\top-\mathscr S(R_2)A_2^\top}
 {L^1_T\B^s_{p,1}}\notag\\
 &\qquad\le C_{\cM}T\left(
 \norm{\delta R}{L^\infty_T\B^s_{p,1}}
 +\norm{\delta A}{L^\infty_T\B^s_{p,1}}\right)\notag\\
 &\qquad\le C_{\cM}T\left(
 \norm{\delta r_0}{\B^s_{p,1}}+\cZ_T\right).
 \label{eq:stress-difference}
\end{align}

Applying Proposition~\ref{prop:stokes} with $\rho=s-1$ and using Proposition~\ref{prop:nonzero-divergence} gives
\begin{equation}\label{eq:abstract-closure}
 \cZ_T
 \le C\norm{\delta u_0}{\B^{s-1}_{p,1}}
 +C_{\cM}T\norm{\delta r_0}{\B^s_{p,1}}
 +C_{\cM}(\eta_T+T)\cZ_T.
\end{equation}
Choose $T$ so that $C_{\cM}(\eta_T+T)\le1/2$.  Then
\begin{equation}\label{eq:lag-stability-general}
 \cZ_T
 \le C_{\cM}\left(
 \norm{\delta u_0}{\B^{s-1}_{p,1}}
 +T\norm{\delta r_0}{\B^s_{p,1}}\right),
\end{equation}
and Proposition~\ref{prop:internal-stability} controls $\delta R$.

\subsection{Uniqueness and continuous dependence}
For equal initial data, \eqref{eq:lag-stability-general} gives $\delta v=0$ and hence $\delta Y=0$, $X_1=X_2$.  Equation \eqref{eq:internal-ode} then gives $R_1=R_2$, and composition with the common inverse flow yields equality in Eulerian variables.  Uniqueness on the whole common lifespan follows by resetting the flow at finitely many times; absolute continuity of the integrable norms guarantees the absorption condition on each subinterval.

Let $(u_{0,n},r_{0,n})\to(u_0,r_0)$.  Proposition~\ref{prop:local-input} supplies a common interval and uniform smallness.  Estimate \eqref{eq:lag-stability-general} gives
\[
 v_n\to v\quad\hbox{in }
 C_T\B^{s-1}_{p,1}\cap L^1_T\B^{s+1}_{p,1},
 \qquad
 R_n\to R\quad\hbox{in }C_T\B^s_{p,1},
\]
and $Y_n\to Y$ in $L^\infty_T(\B^s_{p,1}\cap\B^{s+1}_{p,1})$.  Hence $X_n$ and their inverses converge uniformly with uniform Lipschitz bounds.  Lemma~\ref{lem:composition} yields
\[
 u_n=v_n\circ X_n^{-1}\to v\circ X^{-1}=u
 \quad\hbox{in }C_T\B^{s-1}_{p,1},
\]
\[
 r_n=R_n\circ X_n^{-1}\to R\circ X^{-1}=r
 \quad\hbox{in }C_T\B^s_{p,1}.
\]
For the high velocity norm, put $H_n=(\nabla v_n)A_n$ and $H=(\nabla v)A$.  Then $H_n\to H$ in $L^1_T\B^s_{p,1}$ and
$\nabla u_n=H_n\circ X_n^{-1}$.  The $L^1_t$ part of Lemma~\ref{lem:composition}, applied at the index $s$ to the gradients, gives
\[
 H_n\circ X_n^{-1}\longrightarrow H\circ X^{-1}
 \quad\hbox{in }L^1_T\B^s_{p,1}.
\]
Since these are precisely $\nabla u_n$ and $\nabla u$, the derivative equivalence in homogeneous Besov spaces yields
\[
 u_n\to u\quad\hbox{in }L^1_T\B^{s+1}_{p,1}.
\]
It remains only to pass from the short interval to an arbitrary compact subinterval of the lifespan.  Let $[0,T_0]$ be compactly contained in the maximal lifespan of the limiting solution.  By the absolute continuity of the time-integrable norms, choose a finite partition
\[
 0=t_0<t_1<\cdots<t_N=T_0
\]
such that on each subinterval $[t_k,t_{k+1}]$ the corresponding smallness quantity satisfies
\[
 C_{\cM}\bigl(\eta_{[t_k,t_{k+1}]}+(t_{k+1}-t_k)\bigr)\le \frac12 .
\]
Here $\eta_{[a,b]}$ denotes the analogue of \eqref{eq:eta-general} after restarting the construction at time $a$.  The constants depend only on the uniform critical bounds of the limit solution on $[0,T_0]$.  For initial data sufficiently close to $(u_0,r_0)$, Proposition~\ref{prop:local-input} gives a common lifespan and the same smallness on the first interval.  The estimate above gives convergence at time $t_1$; using those traces as new initial data, the same argument applies on $[t_1,t_2]$.  Iterating over the finite partition proves continuity in the full Eulerian phase space on $[0,T_0]$.  This completes the proof of Theorem~\ref{thm:abstract-main}.

\section{Applications}\label{sec:applications}

\subsection{Viscous non-resistive MHD}
Take $\mathcal V=\mathcal W=\mathbb R^d$, $\bar a=0$, and
\[
 \mathbb S(b)=b\otimes b,
 \qquad \mathbb G(b,M)=Mb.
\]
The law is admissible with $g=0$ and $\mathbb L(b)[M]=Mb$.  The abstract system is \eqref{eq:MHD-intro} without imposing the differential magnetic constraint.  That constraint is not part of the finite-dimensional admissible subspace $\mathcal W$; it is propagated by the frozen-in structure and the Piola identity.

In Lagrangian variables $B=b\circ X$ satisfies
\[
 \partial_tB=(\nabla vA)B,\qquad B(0)=b_0.
\]
Since $\partial_tF=(\nabla vA)F$, uniqueness for this matrix ODE gives the frozen-in formula
\[
 B=Fb_0.
\]
If $\operatorname{div}b_0=0$, then the Piola identity applied to $b=Fb_0\circ X^{-1}$ gives $\operatorname{div}b=0$ in distributions for all times; consequently
$\operatorname{div}(b\otimes b)=(b\cdot\nabla)b$.  The transformed magnetic force may therefore be written either as
$\operatorname{div}((B\otimes B)A^\top)$ or, again using the Piola identity and $\operatorname{div}b_0=0$, as
$(b_0\cdot\nabla)(Fb_0)$.  The abstract theorem proves Theorem~\ref{cor:mhd-main}.  The known critical theory for $p\le2d$ then yields Corollary~\ref{cor:mhd-all-p}.

\subsection{Hookean incompressible viscoelasticity}
Write $U=I+H$ and take $\mathcal V=\mathcal W=\mathbb R^{d\times d}$, $\bar a=0$, with
\[
 \mathbb S(H)=(I+H)(I+H)^\top-I,
 \qquad
 \mathbb G(H,M)=M(I+H).
\]
The right-hand sides are interpreted through the perturbation $H\in\dot B^s_{p,1}$:
\[
 \mathbb S(H)=H+H^\top+HH^\top,
 \qquad
 \mathbb G(H,M)=M+MH.
\]
Thus no constant is assigned a homogeneous Besov norm, and the law is admissible.  The analytic local theory is first proved on the full affine Besov space $I+\dot B^s_{p,1}$; the nonlinear physical constraint manifold is not used in the construction.  Theorem~\ref{thm:abstract-main} gives the analytic part of Theorem~\ref{cor:hookean-main}.

It remains to justify the physical constraint statement without assuming that the regularizing sequence preserves the nonlinear constraints.  We use the exact push-forward formula instead.

\begin{proposition}[Propagation of the Hookean compatibility constraints]\label{prop:hookean-constraints}
Let $(u,U)$ be the critical solution given above and write $U=I+H$.  Assume
\[
 \diver U_0^\top=0,\qquad \det U_0=1,
 \qquad [U_0^{(\alpha)},U_0^{(\beta)}]=0
 \quad(1\le\alpha,\beta\le d).
\]
The Lie brackets are interpreted after separating the affine constants.  Namely, if
$U^{(\alpha)}=e_\alpha+H^{(\alpha)}$, then
\begin{equation}\label{eq:affine-bracket-definition}
 [U^{(\alpha)},U^{(\beta)}]
 :=\partial_\alpha H^{(\beta)}-\partial_\beta H^{(\alpha)}
 +(H^{(\alpha)}\cdot\nabla)H^{(\beta)}
 -(H^{(\beta)}\cdot\nabla)H^{(\alpha)}.
\end{equation}
The linear terms belong to $\dot B^{s-1}_{p,1}$ and the nonlinear terms
are well defined under the columnwise solenoidal condition by
Lemma~\ref{lem:null}.  Then the same three identities hold for every
time in the lifespan.
\end{proposition}

\begin{proof}
Let $X$ be the volume-preserving flow, let $F=I+\D Y$ be the distributional derivative constructed in Proposition~\ref{prop:critical-flow-jacobian}, and set
$\mathcal U=U\circ X$.  The Lagrangian equation is
\[
 \partial_t\mathcal U=(\nabla vA)\mathcal U,
 \qquad \mathcal U(0)=U_0,
\]
whereas $\partial_tF=(\nabla vA)F$ and $F(0)=I$.  Uniqueness for this matrix ODE in the algebra $I+\dot B^s_{p,1}$ gives the exact identity
\begin{equation}\label{eq:hookean-pushforward}
 \mathcal U(t,y)=F(t,y)U_0(y).
\end{equation}
In particular, no approximation of $U_0$ satisfying the nonlinear constraints is needed.

Since $\det F=1$, the algebra property and \eqref{eq:hookean-pushforward} give, in the affine homogeneous algebra,
\[
 (\det U)\circ X=\det(FU_0)=\det U_0=1.
\]
The following distributional computation is justified by the smooth-flow approximation used in Proposition~\ref{prop:critical-flow-jacobian}.  For a column $V=U_0^{(\alpha)}$ and a test function $\varphi$, the change of variables $x=X(t,y)$ and \eqref{eq:hookean-pushforward} yield
\begin{align*}
 \langle\diver U^{(\alpha)}(t),\varphi\rangle
 &=-\int_{\mathbb R^d}F(t,y)V(y)\cdot
      (\nabla_x\varphi)(X(t,y))\,dy\\
 &=-\int_{\mathbb R^d}V(y)\cdot
      \nabla_y(\varphi\circ X)(t,y)\,dy=0.
\end{align*}
Thus $\diver U^\top=0$ in distributions for all times.

We finally prove preservation of commuting columns.  Put
$V_\alpha=U_0^{(\alpha)}=e_\alpha+H_0^{(\alpha)}$.  Choose smooth
Fourier cutoffs $J_n$ commuting with derivatives and set
\[
 V_{\alpha,n}=e_\alpha+J_nH_0^{(\alpha)}.
\]
These are smooth affine vector fields; their perturbative parts are
solenoidal because the columns of $U_0$ are solenoidal.  Their affine
brackets are
\begin{align}\label{eq:initial-affine-bracket-n}
 C_n:=[V_{\alpha,n},V_{\beta,n}]
 ={}&\partial_\alpha J_nH_0^{(\beta)}
    -\partial_\beta J_nH_0^{(\alpha)}\notag\\
 &+(J_nH_0^{(\alpha)}\cdot\nabla)J_nH_0^{(\beta)}
  -(J_nH_0^{(\beta)}\cdot\nabla)J_nH_0^{(\alpha)} .
\end{align}
The linear terms converge directly in $\dot B^{s-1}_{p,1}$, and the
nonlinear terms converge by the middle compensated estimate
\eqref{eq:null-mid}.  Hence
\[
 C_n\longrightarrow [V_\alpha,V_\beta]=0
 \quad\hbox{in }\dot B^{s-1}_{p,1}.
\]
For smooth vector fields the bracket of two divergence-free fields is
divergence free; in particular each $C_n$ is solenoidal.

Let $u_n\to u$ be smooth divergence-free velocities in the critical
solution norms, let $X_n=I+Y_n$ be their volume-preserving flows, and let
$F_n=DX_n$.  Define smooth push-forwards by
\[
 U_n^{(\alpha)}\circ X_n=F_nV_{\alpha,n}.
\]
Naturality of the Lie bracket for smooth diffeomorphisms gives
\begin{equation}\label{eq:bracket-pushforward-smooth}
 [U_n^{(\alpha)},U_n^{(\beta)}]\circ X_n
 =F_nC_n.
\end{equation}
Since $C_n$ is solenoidal, the right-hand side is the compensated Piola
expression
\[
 F_nC_n=C_n+(C_n\cdot\nabla)Y_n.
\]
Lemma~\ref{lem:null}, together with the uniform bound for $Y_n$ in
$C_T\dot B^{s+1}_{p,1}$, yields
\[
 \|F_nC_n\|_{C_T\dot B^{s-1}_{p,1}}
 \lesssim
 (1+\|Y_n\|_{C_T\dot B^{s+1}_{p,1}})
 \|C_n\|_{\dot B^{s-1}_{p,1}}
 \longrightarrow0.
\]
The uniform composition estimate at index $s-1$ therefore gives
\[
 [U_n^{(\alpha)},U_n^{(\beta)}]
 \longrightarrow0
 \quad\hbox{in }C_T\dot B^{s-1}_{p,1}.
\]
On the other hand, $F_nV_{\alpha,n}-e_\alpha$ converges to
$FV_\alpha-e_\alpha$ in $C_T\dot B^s_{p,1}$ by the affine algebra estimates, and the composition lemma at the index $s$ then gives, after subtracting the constant columns,
$H_n^{(\alpha)}:=U_n^{(\alpha)}-e_\alpha\to
H^{(\alpha)}:=U^{(\alpha)}-e_\alpha$ in $C_T\dot B^s_{p,1}$ for every
$\alpha$.  Expanding the affine bracket as in
\eqref{eq:affine-bracket-definition}, the linear derivative terms
converge in $C_T\dot B^{s-1}_{p,1}$, while the nonlinear terms converge
by Lemma~\ref{lem:null}, because the columns are solenoidal.  Thus
\[
 [U_n^{(\alpha)},U_n^{(\beta)}]
 \longrightarrow[U^{(\alpha)},U^{(\beta)}]
 \quad\hbox{in }C_T\dot B^{s-1}_{p,1}.
\]
The limiting bracket is therefore zero.
\end{proof}

The transformed elastic stress is
\[
 (\mathcal U\mathcal U^\top-I)A^\top.
\]
Its constant part has zero divergence by the Piola identity, exactly as required by the affine homogeneous-space convention.

\subsection{Nondiffusive Oldroyd--B}
Let $\mathcal V=\mathbb R^{d\times d}$ and let
$\mathcal W=\operatorname{Sym}_d$.  Set $\bar a=0$ and
\[
 \mathbb S(\tau)=\mu_1\tau,
 \qquad
 \mathbb G(\tau,M)=-a\tau-Q_b(\tau,M)+\mu_2D(M).
\]
The map $Q_b$ is bilinear, the dependence on $M$ is affine, and every term preserves symmetry.  Hence the law is admissible.  Only the linear symmetry constraint is included in $\mathcal W$; positivity or conformation-tensor constraints are nonlinear and are not part of the present admissibility framework.

The Lagrangian stress $T=\tau\circ X$ satisfies the matrix ODE
\begin{equation}\label{eq:oldroyd-lag-ode}
 \partial_tT+aT+Q_b(T,\nabla vA)=\mu_2D(\nabla vA).
\end{equation}
All products in this equation are controlled in the algebra
$\dot B^s_{p,1}$, with each occurrence of $\nabla vA$ understood as
$\nabla v+\nabla v(A-I)$.  For two solutions, writing
$G_i=\nabla v_iA_i$ and subtracting \eqref{eq:oldroyd-lag-ode} gives
\begin{align*}
 \|\delta T(t)\|_{\dot B^s_{p,1}}
 &\le \|\delta\tau_0\|_{\dot B^s_{p,1}}
 +C_{\mathcal M}\int_0^t
   \bigl(1+\|G_1\|_{\dot B^s}+\|G_2\|_{\dot B^s}\bigr)
   \|\delta T\|_{\dot B^s}\,d\theta\\
 &\quad+C_{\mathcal M}\int_0^t
   \|G_1-G_2\|_{\dot B^s}\,d\theta.
\end{align*}
Since
\[
 G_1-G_2=(\nabla\delta v)A_1+(\nabla v_2)(A_1-A_2),
\]
Gronwall's lemma yields
\[
 \|\delta T\|_{L^\infty_T\dot B^s_{p,1}}
 \le C_{\mathcal M}\bigl(
 \|\delta\tau_0\|_{\dot B^s_{p,1}}+\mathcal Z_T\bigr).
\]
The momentum forcing is $\mu_1\operatorname{div}(TA^\top)$, and the affine splitting gives
\[
 T_1A_1^\top-T_2A_2^\top
 =\delta T+\delta T(A_1-I)^\top+T_2(A_1-A_2)^\top.
\]
Hence
\[
 \|T_1A_1^\top-T_2A_2^\top\|_{L^1_T\dot B^s_{p,1}}
 \le C_{\mathcal M}T
 \bigl(\|\delta\tau_0\|_{\dot B^s_{p,1}}+\mathcal Z_T\bigr).
\]
Thus the abstract closure estimate applies without any additional negative-index product.  This proves Theorem~\ref{cor:oldroyd-main}.  More precisely, the local critical theorem in \cite{DeAnnaPaicu} for the principal source-free model in dimensions $d\ge3$ assumes $p<2d$; our conclusion covers $p\ge2d$ and also allows the damping, source, and non-corotational objective terms in \eqref{eq:Oldroyd-intro}.  We make no claim here about the global large-data or critical-Lorentz results of \cite{DeAnnaPaicu}.  The restriction to the Besov summability index $1$ is consistent with the ill-posedness for $q>1$ proved in \cite{LiYuZhu}.

\subsection{Why the three applications are not repetitions}
In MHD, the internal variable is a solenoidal vector and admits an explicit frozen-in formula.  In Hookean elasticity, the internal variable is an affine matrix field and the physical solution lies on a nonlinear invariant constraint manifold.  In Oldroyd--B, the stress is neither solenoidal nor a deformation gradient and instead solves an objective matrix ODE.  The common part of the analysis is therefore exactly the rough Piola geometry of the velocity, while the model-specific evolution is isolated in the algebra-level Lagrangian ODE.  This separation is the content of the abstract principle rather than a finite repetition of the MHD proof.

\section{Limitations and further directions}
The mechanism uses three structural ingredients: a volume-preserving flow, parabolic maximal regularity for the velocity, and a constitutive evolution that becomes an ODE in $\B^s_{p,1}$ after composition.  It therefore extends directly to finite systems of frozen-in vector or matrix fields and to other smooth objective stress laws whose dependence on the velocity gradient is affine in the sense of Definition~\ref{def:admissible}.

The endpoint $p=\infty$ is not included, whereas the finite threshold $p=2d$ is covered by the same estimates.  Related sharp ill-posedness phenomena for nondiffusive MHD further emphasize the delicacy of endpoint regularity; see \cite{ChenNieYe}.  In the proof of Lemma~\ref{lem:null}, the high--high kernel is $2^{-2s(k-j)}$ and loses summability when $s=0$.  A logarithmically corrected endpoint or a frequency-envelope version would require an additional idea.  Similarly, the present proof is intrinsically tied to the $\ell^1$ Besov summation; known norm-inflation results for $q>1$ indicate that this restriction is not merely technical in the nondiffusive models considered here.

Bounded domains and free boundaries require a boundary-compatible version of the compensated product, solenoidal extension operators, and Stokes maximal regularity with nonhomogeneous boundary data.  Compressible systems require replacing $A$ by the cofactor $J A$ and controlling the Jacobian.  These directions lie beyond the present paper.

\appendix
\section{Local construction for admissible laws}\label{sec:existence-appendix}

We give the details needed for Proposition~\ref{prop:local-input}.  Only the existence part is used before uniqueness is established, so compactness rather than a contraction argument is sufficient.

\subsection{Smooth parabolic approximation}
Let $J_n$ be a smooth self-adjoint Fourier cutoff to
$\{2^{-n}\le |\xi|\le2^n\}$, commuting with derivatives and the Leray projector, and let $\varepsilon_n\downarrow0$.  We consider
\begin{equation}\label{eq:approx-system}
\left\{
\begin{aligned}
 \partial_tu_n-\nu\Delta u_n
 &=-\mathbb P(u_n\cdot\nabla u_n)
   +\mathbb P\operatorname{div}\mathscr S(r_n),\\
 \partial_tr_n+u_n\cdot\nabla r_n-\varepsilon_n\Delta r_n
 &=\mathscr G(r_n,\nabla u_n),\\
 (u_n,r_n)|_{t=0}&=(J_nu_0,J_nr_0).
\end{aligned}
\right.
\end{equation}
For each fixed $n$, the initial data are smooth and the second equation has a positive artificial diffusivity.  Standard quasilinear parabolic theory therefore gives a maximal smooth solution; see, for example, \cite{Amann}.  This regularization is used only for construction: all estimates below are independent of $\varepsilon_n$, and the artificial diffusion disappears in the limit.  The velocity remains solenoidal, and the finite-dimensional constraint $r_n(t,x)\in\mathcal W$ is preserved because the heat operator and the constitutive vector field both leave $\mathcal W$ invariant.

\subsection{Uniform critical estimates}
The large-data construction is organized around the heat flow of the initial velocity.  Set
\[
 u_{L,n}(t)=e^{\nu t\Delta}J_nu_0,
 \qquad w_n=u_n-u_{L,n}.
\]
Then $w_n(0)=0$ and
\begin{equation}\label{eq:w-equation}
 \partial_tw_n-\nu\Delta w_n
 =-\mathbb P(u_n\cdot\nabla u_n)
 +\mathbb P\operatorname{div}\mathscr S(r_n).
\end{equation}
For $0<T$ define
\[
 \Lambda_n(T):=
 \norm{u_{L,n}}{L^1_T\B^{s+1}_{p,1}}
 +\norm{u_{L,n}}{L^2_T\B^s_{p,1}},
\]
\[
 W_n^\infty(T):=\norm{w_n}{L^\infty_T\B^{s-1}_{p,1}},
 \qquad
 W_n^1(T):=\norm{w_n}{L^1_T\B^{s+1}_{p,1}},
 \qquad
 R_n(T):=\norm{r_n}{L^\infty_T\B^s_{p,1}}.
\]
The heat estimates give
\begin{align}\label{eq:linear-heat-small-general}
 \sup_n\Lambda_n(T)\longrightarrow0\qquad(T\downarrow0).
\end{align}
Indeed the $L^1_T\dot B^{s+1}_{p,1}$ part follows from
\begin{align}\label{eq:homogeneous-low-heat-small}
 &\left\|e^{\nu t\Delta}\sum_{j\le N}\dot\Delta_j u_0\right\|_{L^1_T\dot B^{s+1}_{p,1}}\notag\\
 &\qquad\lesssim
 \sum_{j\le N}2^{j(s-1)}
 \bigl(1-e^{-c\nu T2^{2j}}\bigr)
 \|\dot\Delta_j u_0\|_{L^p}
 \longrightarrow0
 \qquad(T\downarrow0),
\end{align}
combined with smallness of the high-frequency tail.  The $L^2_T\dot B^s_{p,1}$ part is identical, with the factor
$\bigl(1-e^{-c\nu T2^{2j}}\bigr)^{1/2}$ in place of
$1-e^{-c\nu T2^{2j}}$.  Notice that the sum $j\le N$ is still infinite in the homogeneous decomposition; the convergence is a dominated-convergence statement in the Besov coefficient sequence.

Since $\diver u_n=0$,
\[
 u_n\cdot\nabla u_n=\diver(u_n\otimes u_n).
\]
The algebra property of $\dot B^s_{p,1}$ and interpolation give
\begin{align}
 \norm{u_n\cdot\nabla u_n}{L^1_T\B^{s-1}_{p,1}}
 &\lesssim \norm{u_n}{L^2_T\B^s_{p,1}}^2\notag\\
 &\lesssim \Lambda_n(T)^2+W_n^\infty(T)W_n^1(T).
 \label{eq:approx-convection}
\end{align}
Here and below constants may depend on $\nu$, which is fixed.  If $R_n(T)\le R$, the tame stress estimate gives
\begin{equation}\label{eq:approx-stress}
 \norm{\mathscr S(r_n)}{L^1_T\B^s_{p,1}}
 \le C_R T R_n(T).
\end{equation}
Heat maximal regularity applied to the zero-initial-data equation
\eqref{eq:w-equation} yields
\begin{align}
 W_n^\infty(T)+W_n^1(T)
 \le{}&C\Lambda_n(T)^2
 +C W_n^\infty(T)W_n^1(T)
 +C C_R T R_n(T).
 \label{eq:exist-w}
\end{align}
This is the point where the heat-flow splitting is essential: no term of the form
$\|u_0\|_{\dot B^{s-1}}W_n^1$ has to be absorbed.

Let
\[
 V_n(t)=\int_0^t\norm{u_n(\tau)}{\B^{s+1}_{p,1}}\dd\tau
 \le \Lambda_n(t)+W_n^1(t).
\]
Lemma~\ref{lem:endpoint-transport}, together with \eqref{eq:G-single}, gives
\begin{align}
 R_n(t)+c\varepsilon_n\|r_n\|_{L^1(0,t;\dot B^{s+2}_{p,1})}
 \le{}&C e^{CV_n(t)}\Biggl[
 \|J_nr_0\|_{\dot B^s_{p,1}}\notag\\
 &\quad+C_R\int_0^t
 \left(\|r_n(\tau)\|_{\dot B^s_{p,1}}
 +\|u_n(\tau)\|_{\dot B^{s+1}_{p,1}}\right)\,d\tau\Biggr].
 \label{eq:exist-r}
\end{align}
The constants are independent of $\varepsilon_n$.

We now close the large-data bootstrap.  The order of choices is important: first fix the internal radius, then the zero-initial-data velocity radius, and only then choose the lifespan.  Fix
\[
 R_*:=2C\bigl(1+\norm{r_0}{\B^s_{p,1}}\bigr),
\]
where $C$ is large enough to dominate the constant in \eqref{eq:exist-r} for
$V_n\le1$.  Choose a small number $\rho>0$, depending only on the universal constants and $\nu$, so that $C\rho^2\le\rho/8$.  By \eqref{eq:linear-heat-small-general} choose $T>0$ so small that, uniformly in $n$,
\[
 C\Lambda_n(T)^2+C C_{R_*}T R_*\le \rho/4,
 \qquad
 \Lambda_n(T)+\rho\le1,
\]
and also so that the right-hand side of \eqref{eq:exist-r} is at most
$R_*/2$ whenever $R_n\le R_*$ and $W_n^\infty+W_n^1\le\rho$.  On the maximal subinterval on which
\[
 R_n\le R_*,\qquad W_n^\infty+W_n^1\le\rho,
\]
estimates \eqref{eq:exist-w} and \eqref{eq:exist-r} make both inequalities strict.  Hence the bootstrap interval is $[0,T]$.  Consequently
\begin{equation}\label{eq:uniform-existence-bounds}
 \sup_n\left(
 \norm{u_n}{L^\infty_T\B^{s-1}_{p,1}}
 +\norm{u_n}{L^1_T\B^{s+1}_{p,1}}
 +\norm{u_n}{L^2_T\B^s_{p,1}}
 +\norm{r_n}{L^\infty_T\B^s_{p,1}}
 \right)<\infty.
\end{equation}
Indeed $u_n=u_{L,n}+w_n$, the heat flow is bounded in
$L^\infty_T\dot B^{s-1}_{p,1}$ by the initial norm, and the time-integrable parts are bounded by $\Lambda_n+W_n^1$.

The pressure is recovered from
\begin{equation}\label{eq:approx-pressure}
 -\Delta P_n=\partial_i\partial_j
 \bigl((u_n)_i(u_n)_j-\mathscr S(r_n)_{ij}\bigr).
\end{equation}
Riesz-transform bounds, \eqref{eq:approx-convection}, and
\eqref{eq:approx-stress} give
\begin{equation}\label{eq:approx-pressure-bound}
 \norm{P_n}{L^1_T\B^s_{p,1}}
 \lesssim
 \norm{u_n}{L^2_T\B^s_{p,1}}^2
 +\norm{\mathscr S(r_n)}{L^1_T\B^s_{p,1}}.
\end{equation}
The momentum equation then gives
\begin{equation}\label{eq:approx-material-bound}
 \sup_n\norm{D_tu_n}{L^1_T\B^{s-1}_{p,1}}<\infty.
\end{equation}
The internal-variable equation, \eqref{eq:G-single}, and the dissipative term in \eqref{eq:exist-r} give
\begin{equation}\label{eq:approx-internal-material}
 \sup_n\|\partial_tr_n+u_n\cdot\nabla r_n\|_{L^1_T\dot B^s_{p,1}}<\infty.
\end{equation}
For compactness it is also useful that, on every compact spatial set and for every $0<\varepsilon<s$,
\[
 \partial_tu_n\quad\hbox{is bounded in }L^1_T\B^{s-1}_{p,1},
 \qquad
 \partial_tr_n\quad\hbox{is bounded in }L^1_T\B^{s-1-\varepsilon}_{p,1,\mathrm{loc}}.
\]
The second assertion follows from the transport product estimate with
$u_n\in L^1_T\dot B^{s+1}_{p,1}\hookrightarrow L^1_T\mathrm{Lip}$, the uniform bound for $r_n$ in $\dot B^s_{p,1}$, the bound for the source term, and
$\varepsilon_n\|r_n\|_{L^1_T\dot B^{s+2}_{p,1}}\lesssim1$.

\subsection{Compactness and passage to the limit}
We give the compactness argument in enough detail to identify all nonlinear terms.  Fix a compactly supported cutoff $\chi$ and choose $0<\varepsilon<s$.  Multiplication by $\chi$ maps the homogeneous spaces continuously into the corresponding local inhomogeneous spaces, and on a fixed compact support the embeddings
\[
 B^{\alpha}_{p,1}\Subset B^{\alpha-\varepsilon}_{p,1}
\]
are compact.  From the uniform bounds above and the equations, one has
\[
 \partial_tu_n\quad\hbox{bounded in }L^1_T\dot B^{s-1}_{p,1},
 \qquad
 \partial_tr_n\quad\hbox{bounded in }
 L^1_T\dot B^{s-1-\varepsilon}_{p,1,\mathrm{loc}}.
\]
Simon's compactness theorem \cite{SimonCompactness}, followed by a diagonal extraction over cutoffs, gives
\begin{align}
 u_n&\longrightarrow u
 &&\text{in }L^1(0,T;\dot B^{s+1-\varepsilon}_{p,1,\mathrm{loc}})
 \cap L^2(0,T;\dot B^{s-\varepsilon}_{p,1,\mathrm{loc}}),
 \label{eq:compact-u}\\
 r_n&\longrightarrow r
 &&\text{in }L^q(0,T;\dot B^{s-\varepsilon}_{p,1,\mathrm{loc}})
 \quad\text{for every finite }q.
 \label{eq:compact-r}
\end{align}
We do not use, and do not claim, strong compactness in $C_t$ at this stage; the time continuity in the critical norms is recovered from the limiting equations after the passage to the limit.  Since $s-\varepsilon>0$, \eqref{eq:compact-u} implies
\[
 \nabla u_n\longrightarrow \nabla u
 \quad\hbox{in }L^1(0,T;L^p_{\mathrm{loc}}).
\]
Up to a further subsequence, $r_n\to r$ almost everywhere, and the uniform embedding
$\dot B^s_{p,1}\hookrightarrow L^\infty$ supplies a common pointwise bound.  Since
$\mathcal W$ is a finite-dimensional closed linear subspace and each $r_n$ is
$\mathcal W$-valued, the limit satisfies $r(t,x)\in\mathcal W$ for almost every
$(t,x)$.

We now identify the equations.  Since $u_n$ is solenoidal, the transport terms may be tested in divergence form.  The strong local convergence of $u_n$ in $L^2_TL^p_{\mathrm{loc}}$, together with the uniform bound of $u_n$ in $L^2_T\dot B^s_{p,1}\hookrightarrow L^2_TL^\infty$, gives
\[
 u_n\otimes u_n\longrightarrow u\otimes u
 \quad\hbox{in }L^1_{\mathrm{loc}}((0,T)\times\mathbb R^d).
\]
Moreover, $r_n$ is uniformly bounded in $L^\infty$ and converges almost everywhere, while $u_n\to u$ in $L^1_{\mathrm{loc}}$; hence
\[
 r_nu_n\longrightarrow ru
 \quad\hbox{in }L^1_{\mathrm{loc}}((0,T)\times\mathbb R^d).
\]
Thus both $u_n\cdot\nabla u_n$ and $u_n\cdot\nabla r_n$ converge in distributions.  Smooth finite-dimensional composition and dominated convergence imply
\[
 \mathscr S(r_n)\longrightarrow\mathscr S(r)
 \quad\hbox{locally in }L^q
\]
for every finite $q$.  For the constitutive source, use the affine structure
\[
 \mathscr G(r_n,\nabla u_n)
 =g(\bar a+r_n)+\mathbb L(\bar a+r_n)[\nabla u_n].
\]
The first term converges by dominated convergence.  In the second, split
\begin{align*}
 &\mathbb L(\bar a+r_n)[\nabla u_n]
 -\mathbb L(\bar a+r)[\nabla u]\\
 &\quad=\mathbb L(\bar a+r_n)[\nabla(u_n-u)]
 +\bigl(\mathbb L(\bar a+r_n)-\mathbb L(\bar a+r)\bigr)[\nabla u].
\end{align*}
The first term tends to zero in $L^1_TL^p_{\mathrm{loc}}$ by the strong convergence of $\nabla u_n$ and the uniform $L^\infty$ bound on the coefficient.  For the second term, on each compact spatial set $K$ the $L^p(K)$ norm tends to zero for almost every time by dominated convergence in space, and it is bounded by
$C\|\nabla u(t)\|_{L^p(K)}$, which is integrable in time.  Hence it also tends to zero by dominated convergence in time.  Thus the constitutive source converges in distributions.

The artificial diffusion disappears in a fixed global negative norm:
\[
 \|\varepsilon_n\Delta r_n\|_{L^\infty_T\dot B^{s-2}_{p,1}}
 \le \varepsilon_n\|r_n\|_{L^\infty_T\dot B^s_{p,1}}
 \longrightarrow0.
\]
Consequently $(u,r)$ solves \eqref{eq:general-system}.  The initial trace is inherited from the approximations: integrating against smooth space--time test functions and passing to the limit identifies the trace at $t=0$ with
$(u_0,r_0)$, because $J_nu_0\to u_0$ in $\dot B^{s-1}_{p,1}$ and
$J_nr_0\to r_0$ in $\dot B^s_{p,1}$.  Lemma~\ref{lem:bochner-besov-fatou}, used only with the exponents $q=\infty$ and $q=2$, gives
\[
 u\in L^\infty_T\dot B^{s-1}_{p,1}\cap L^2_T\dot B^s_{p,1},
 \qquad r\in L^\infty_T\dot B^s_{p,1}.
\]
The endpoint $q=1$ of Lemma~\ref{lem:bochner-besov-fatou} is not invoked here.  Instead, the pressure estimate applied to the limiting equation gives
$P\in L^1_T\dot B^s_{p,1}$, because
$u\otimes u\in L^1_T\dot B^s_{p,1}$ and
$\mathscr S(r)\in L^1_T\dot B^s_{p,1}$.  The Duhamel formula, or equivalently the standard heat maximal-regularity estimate with forcing in
$L^1_T\dot B^{s-1}_{p,1}$, then yields
$u\in C_T\dot B^{s-1}_{p,1}\cap L^1_T\dot B^{s+1}_{p,1}$.  Applying Lemma~\ref{lem:endpoint-transport} to the limiting internal equation gives
$r\in C_T\dot B^s_{p,1}$.  Finally,
\eqref{eq:pressure-class} follows from the same pressure estimate,
$D_tu\in L^1_T\dot B^{s-1}_{p,1}$ follows from the momentum equation, and
\eqref{eq:material-class} follows from Lemma~\ref{lem:constitutive-tame} applied to the limiting constitutive source.

\subsection{Common lifespan for convergent data}
Suppose $(u_{0,n},r_{0,n})\to(u_0,r_0)$ in the critical phase space.  The internal data remain in one fixed $\dot B^s_{p,1}$ ball.  The only point requiring uniformity is the smallness of the linear heat flows.  Given $\delta>0$, choose one dyadic index $N$ such that, for all sufficiently large $n$,
\[
 \sum_{j>N}2^{j(s-1)}
 \|\dot\Delta_ju_{0,n}\|_{L^p}<\delta;
\]
this follows from convergence in the $\ell^1$ Besov norm.  The heat estimate gives a uniform contribution bounded by $C\delta$ for this high-frequency tail in both
$L^1_T\dot B^{s+1}_{p,1}$ and $L^2_T\dot B^s_{p,1}$.  For the low-frequency part $j\le N$, use \eqref{eq:homogeneous-low-heat-small} and its $L^2$ analogue.  Although this is still an infinite homogeneous tail, the convergence is dominated convergence in the $\ell^1$ sequence of Besov coefficients.  Uniformity in $n$ follows by writing
$u_{0,n}=u_0+(u_{0,n}-u_0)$: the heat map is uniformly bounded from
$\dot B^{s-1}_{p,1}$ to the two heat norms, while the fixed datum $u_0$ is handled by dominated convergence.  Therefore
\[
 \sup_{n\ge n_0}\left(
 \|e^{\nu t\Delta}u_{0,n}\|_{L^1_T\dot B^{s+1}_{p,1}}
 +\|e^{\nu t\Delta}u_{0,n}\|_{L^2_T\dot B^s_{p,1}}
 \right)\longrightarrow0
 \qquad(T\downarrow0).
\]
Applying the perturbative bootstrap for $u_n=e^{\nu t\Delta}u_{0,n}+w_n$ with the same constants $R_*$ and $\rho$ gives a common positive lifespan and uniform critical bounds for all large $n$.  This proves the common-lifespan part of Proposition~\ref{prop:local-input}.

We finally record the uniform absolute continuity required later.  Reapplying the same zero-initial-data estimate on $(0,t)$ gives
\[
 \sup_{n\ge n_0}\bigl(W_n^\infty(t)+W_n^1(t)\bigr)\longrightarrow0
 \qquad(t\downarrow0),
\]
because the corresponding linear heat norms and the factor $C_{R_*}t$ tend to zero uniformly.  Consequently
\[
 \sup_{n\ge n_0}\left(
 \|u_n\|_{L^1(0,t;\dot B^{s+1}_{p,1})}
 +\|u_n\|_{L^2(0,t;\dot B^s_{p,1})}
 \right)\longrightarrow0
 \qquad(t\downarrow0).
\]
The stress contribution is bounded by $C_{R_*}t$, the pressure estimate follows from
\eqref{eq:approx-pressure-bound}, and the momentum equation gives the analogous assertion for $D_tu_n$.  The internal material derivative is controlled by
$C_{R_*}(t+\|u_n\|_{L^1(0,t;\dot B^{s+1})})$.  Hence every time-integrable quantity used in the stability argument is uniformly absolutely continuous at the initial time.

All estimates in this appendix are obtained before uniqueness is invoked.  After Section~\ref{sec:stability} proves uniqueness, every compactness subsequence has the same limit; this identifies the solution family without introducing a circular dependence into the existence or lifespan argument.

\subsection{Continuation}
Let $(u,r)$ be a solution on $[0,T^*)$ satisfying the bound in Theorem~\ref{thm:abstract-main}.  Interpolation gives $u\in L^2(0,T^*;\dot B^s_{p,1})$.  The equation and the pressure estimate then imply
\[
 \partial_tu\in L^1(0,T^*;\dot B^{s-1}_{p,1}),
\]
so $u(t)$ has a limit $u_*$ in $\dot B^{s-1}_{p,1}$ as $t\uparrow T^*$.  Hence, for $t_0$ sufficiently close to $T^*$, the heat flows issued from $u(t_0)$ have a common smallness bound:
\[
 \|e^{\nu(t-t_0)\Delta}u(t_0)\|_{L^1(t_0,t_0+\tau;\dot B^{s+1}_{p,1})}
 +\|e^{\nu(t-t_0)\Delta}u(t_0)\|_{L^2(t_0,t_0+\tau;\dot B^s_{p,1})}
\]
is uniformly small for one $\tau>0$ independent of such $t_0$.  This follows by comparing $u(t_0)$ with the fixed limit $u_*$ and using the low--high heat estimate above.  The internal data $r(t_0)$ stay in one fixed bounded subset of $\dot B^s_{p,1}$.  Restarting the perturbative construction at $t_0$, with
$u_L(t)=e^{\nu(t-t_0)\Delta}u(t_0)$ and $u=u_L+w$, gives a lifespan bounded from below independently of $t_0$.  By uniqueness, the restarted solution coincides with the original one on the overlap $[t_0,T^*)$.  Choosing $t_0$ so close to $T^*$ that $t_0+\tau>T^*$ therefore extends the solution beyond $T^*$ and proves the continuation assertion.

\section{Vector-valued BV calculus for the Piola graph}\label{sec:bv-appendix}
We record the functional-analytic input used in Lemma~\ref{lem:time-diff}.  The details are included to make clear why the operator domain must be the solenoidal subspace.  We use only the elementary vector-measure product rule, together with the Radon--Nikodym property of the Besov space; no further fine structure of Besov spaces is needed here.

\begin{proposition}[Radon--Nikodym property and BV product rule]\label{prop:bv-product-rule}
Let $1<p<\infty$ and $\gamma\in\mathbb R$.
\begin{enumerate}[label=\textup{(\roman*)}]
 \item The homogeneous Besov space $\dot B^\gamma_{p,1}(\mathbb R^d)$, realized in $\mathcal S'_h$, has the Radon--Nikodym property.  Every closed subspace, in particular the solenoidal subspace, has the same property.
 \item Let $E_0$ have the Radon--Nikodym property, let $E_1$ be a Banach space, let $z\in BV(0,T;E_0)$, and let
 $L\in W^{1,1}(0,T;\mathcal L(E_0,E_1))$.  Then
 \begin{equation}\label{eq:bv-product-rule}
  D(Lz)=L\,Dz+L'z\,dt
 \end{equation}
 as $E_1$-valued Radon measures.
\end{enumerate}
\end{proposition}

\begin{proof}
For the first assertion, choose Littlewood--Paley cutoffs
$\dot\Delta_j$ and enlarged cutoffs $\widetilde{\dot\Delta}_j$ satisfying the usual reconstruction identity.  The analysis map
\[
 \mathcal A f=\bigl(2^{j\gamma}\dot\Delta_jf\bigr)_{j\in\mathbb Z}
\]
is an isomorphic embedding of $\dot B^\gamma_{p,1}$ into
$\ell^1(\mathbb Z;L^p)$.  The synthesis map
\[
 \mathcal R(g_j)_j
 =\sum_{j\in\mathbb Z}2^{-j\gamma}
   \widetilde{\dot\Delta}_jg_j
\]
is bounded from $\ell^1(L^p)$ to $\dot B^\gamma_{p,1}$ and satisfies
$\mathcal R\mathcal A=I$.  Hence the Besov space is isomorphic to a complemented, and therefore closed, subspace of $\ell^1(L^p)$.  Since $L^p$ has the Radon--Nikodym property for $1<p<\infty$, its $\ell^1$-sum has that property, and closed subspaces inherit it; see \cite{DiestelUhl}.

For the second assertion, use the continuous representative of $L$.  Since
$E_0$ has the Radon--Nikodym property, the $E_0$-valued measure $Dz$ admits the polar representation
\[
 Dz=\xi\,|Dz|,
 \qquad \xi(t)\in E_0,\quad \|\xi(t)\|_{E_0}=1
 \quad |Dz|\hbox{-a.e.}
\]
We define
\[
 L\,Dz:=L(t)\xi(t)\,|Dz|.
\]
For step functions $z$ and smooth operator paths $L$, formula
\eqref{eq:bv-product-rule} follows by summing the jumps and applying the ordinary product rule on the intervals of constancy.  General Banach-valued $BV$ functions are obtained by strict approximation, and $W^{1,1}$ operator paths are approximated in $W^{1,1}$ by smooth paths.  Passing to the limit in the vector measures yields \eqref{eq:bv-product-rule}; equivalently, this is the standard vector-measure product rule in \cite{DiestelUhl}.
\end{proof}

In Lemma~\ref{lem:time-diff}, $E_0=E_\sigma$ and $E_1=E$.  Because $z(t)$ belongs to $E_\sigma$ for every $t$, the quotient map $E\to E/E_\sigma$ annihilates $z$, and hence also annihilates the measure $Dz$.  Thus the polar vector of the singular part lies in $E_\sigma$ almost everywhere.  This is exactly what permits the coercivity of $L_Y$ on the solenoidal domain to eliminate the singular derivative.


\begin{thebibliography}{99}

\bibitem{Amann}
H.~Amann,
\emph{Linear and Quasilinear Parabolic Problems. Vol. I},
Birkh\"auser, Basel, 1995.

\bibitem{BCD}
H.~Bahouri, J.-Y.~Chemin and R.~Danchin,
\emph{Fourier Analysis and Nonlinear Partial Differential Equations},
Grundlehren Math. Wiss. 343, Springer, Heidelberg, 2011.

\bibitem{BerghLofstrom}
J.~Bergh and J.~L\"ofstr\"om,
\emph{Interpolation Spaces. An Introduction},
Grundlehren Math. Wiss. 223, Springer-Verlag, Berlin--New York, 1976.

\bibitem{BirdArmstrongHassager}
R.~B.~Bird, R.~C.~Armstrong and O.~Hassager,
\emph{Dynamics of Polymeric Liquids. Vol. 1: Fluid Mechanics}, 2nd ed.,
Wiley, New York, 1987.

\bibitem{BoffiBrezziFortin}
D.~Boffi, F.~Brezzi and M.~Fortin,
\emph{Mixed Finite Element Methods and Applications},
Springer Series in Computational Mathematics 44, Springer, Heidelberg, 2013.

\bibitem{CheminMasmoudi}
J.-Y.~Chemin and N.~Masmoudi,
About lifespan of regular solutions of equations related to viscoelastic fluids,
\emph{SIAM J. Math. Anal.} \textbf{33} (2001), 84--112.

\bibitem{CheminBook}
J.-Y.~Chemin,
\emph{Perfect Incompressible Fluids},
Oxford Lecture Series in Mathematics and its Applications 14,
Oxford University Press, New York, 1998.

\bibitem{CheminEtAl}
J.-Y.~Chemin, D.~S.~McCormick, J.~C.~Robinson and J.~L.~Rodrigo,
Local existence for the non-resistive MHD equations in Besov spaces,
\emph{Adv. Math.} \textbf{286} (2016), 1--31, doi:10.1016/j.aim.2015.09.004.

\bibitem{ChenNieYe}
Q.~Chen, Y.~Nie and W.~Ye,
Sharp ill-posedness for the non-resistive MHD equations in Sobolev spaces,
\emph{J. Funct. Anal.} \textbf{286} (2024), no.~6, Paper No.~110302, doi:10.1016/j.jfa.2023.110302.

\bibitem{ChenHao}
Q.~Chen and X.~Hao,
Global well-posedness in the critical Besov spaces for the incompressible Oldroyd--B model without damping mechanism,
\emph{J. Math. Fluid Mech.} \textbf{21} (2019), Paper No.~42, 23 pp.

\bibitem{ChenMiao}
Q.~Chen and C.~Miao,
Global well-posedness of viscoelastic fluids of Oldroyd type in Besov spaces,
\emph{Nonlinear Anal.} \textbf{68} (2008), no.~7, 1928--1939.

\bibitem{ConstantinKliegl}
P.~Constantin and M.~Kliegl,
Note on global regularity for two-dimensional Oldroyd--B fluids with diffusive stress,
\emph{Arch. Ration. Mech. Anal.} \textbf{206} (2012), 725--740.

\bibitem{ConstantinLE}
P.~Constantin,
Lagrangian--Eulerian methods for uniqueness in hydrodynamic systems,
\emph{Adv. Math.} \textbf{278} (2015), 67--102.

\bibitem{DanchinMucha}
R.~Danchin and P.~B.~Mucha,
A Lagrangian approach for the incompressible Navier--Stokes equations with variable density,
\emph{Comm. Pure Appl. Math.} \textbf{65} (2012), 1458--1480.

\bibitem{DanchinTransport}
R.~Danchin,
Estimates in Besov spaces for transport and transport-diffusion equations with almost Lipschitz coefficients,
\emph{Rev. Mat. Iberoam.} \textbf{21} (2005), no.~3, 863--888.

\bibitem{DeAnnaPaicu}
F.~De Anna and M.~Paicu,
The Fujita--Kato theorem for some Oldroyd--B model,
\emph{J. Funct. Anal.} \textbf{279} (2020), no.~11, Paper No.~108761, 64 pp., doi:10.1016/j.jfa.2020.108761.

\bibitem{DiestelUhl}
J.~Diestel and J.~J.~Uhl, Jr.,
\emph{Vector Measures},
Math. Surveys 15, American Mathematical Society, Providence, RI, 1977.

\bibitem{ElgindiRousset}
T.~M.~Elgindi and F.~Rousset,
Global regularity for some Oldroyd--B type models,
\emph{Comm. Pure Appl. Math.} \textbf{68} (2015), 2005--2021.

\bibitem{ElgindiLiu}
T.~M.~Elgindi and J.~Liu,
Global wellposedness to the generalized Oldroyd type models in $\mathbb R^3$,
\emph{J. Differential Equations} \textbf{259} (2015), no.~5, 1958--1966.

\bibitem{FeffermanEtAlJFA}
C.~L.~Fefferman, D.~S.~McCormick, J.~C.~Robinson and J.~L.~Rodrigo,
Higher order commutator estimates and local existence for the non-resistive MHD equations and related models,
\emph{J. Funct. Anal.} \textbf{267} (2014), no.~4, 1035--1056, doi:10.1016/j.jfa.2014.03.021.

\bibitem{FeffermanEtAlARMA}
C.~L.~Fefferman, D.~S.~McCormick, J.~C.~Robinson and J.~L.~Rodrigo,
Local existence for the non-resistive MHD equations in nearly optimal Sobolev spaces,
\emph{Arch. Ration. Mech. Anal.} \textbf{223} (2017), 677--691, doi:10.1007/s00205-016-1042-7.

\bibitem{GigaSohr}
Y.~Giga and H.~Sohr,
Abstract $L^p$ estimates for the Cauchy problem with applications to the Navier--Stokes equations in exterior domains,
\emph{J. Funct. Anal.} \textbf{102} (1991), 72--94.

\bibitem{GiraultRaviart}
V.~Girault and P.-A.~Raviart,
\emph{Finite Element Methods for Navier--Stokes Equations: Theory and Algorithms},
Springer Series in Computational Mathematics 5, Springer-Verlag, Berlin, 1986.

\bibitem{GuillopeSaut}
C.~Guillop\'e and J.-C.~Saut,
Existence results for the flow of viscoelastic fluids with a differential constitutive law,
\emph{Nonlinear Anal.} \textbf{15} (1990), no.~9, 849--869.

\bibitem{LeiLiuZhou}
Z.~Lei, C.~Liu and Y.~Zhou,
Global solutions for incompressible viscoelastic fluids,
\emph{Arch. Ration. Mech. Anal.} \textbf{188} (2008), 371--398.

\bibitem{LiTanYin}
J.~Li, W.~Tan and Z.~Yin,
Local existence and uniqueness for the non-resistive MHD equations in homogeneous Besov spaces,
\emph{Adv. Math.} \textbf{317} (2017), 786--798, doi:10.1016/j.aim.2017.07.013.

\bibitem{LiYuZhu}
J.~Li, Y.~Yu and W.~Zhu,
Ill-posedness issue on the Oldroyd--B model in the critical Besov spaces,
\emph{Proc. Roy. Soc. Edinburgh Sect. A Math.}, First View (2025), 1--20,
doi:10.1017/prm.2025.10057.

\bibitem{LinLiuZhang}
F.-H.~Lin, C.~Liu and P.~Zhang,
On hydrodynamics of viscoelastic fluids,
\emph{Comm. Pure Appl. Math.} \textbf{58} (2005), 1437--1471.

\bibitem{LionsMasmoudi}
P.-L.~Lions and N.~Masmoudi,
Global solutions for some Oldroyd models of non-Newtonian flows,
\emph{Chinese Ann. Math. Ser. B} \textbf{21} (2000), 131--146.

\bibitem{Meyer}
Y.~Meyer,
Wavelets, paraproducts, and Navier--Stokes equations,
in \emph{Current Developments in Mathematics, 1996},
International Press, Boston, MA, 1997, 105--212.

\bibitem{Oldroyd1950}
J.~G.~Oldroyd,
On the formulation of rheological equations of state,
\emph{Proc. Roy. Soc. London Ser. A} \textbf{200} (1950), no.~1063, 523--541.

\bibitem{Qian}
J.~Qian,
Well-posedness in critical spaces for incompressible viscoelastic fluid system,
\emph{Nonlinear Anal.} \textbf{72} (2010), no.~6, 3222--3234.

\bibitem{Renardy}
M.~Renardy,
\emph{Mathematical Analysis of Viscoelastic Flows},
CBMS-NSF Regional Conference Series in Applied Mathematics 73, SIAM, Philadelphia, 2000.

\bibitem{RunstSickel}
T.~Runst and W.~Sickel,
\emph{Sobolev Spaces of Fractional Order, Nemytskij Operators, and Nonlinear Partial Differential Equations},
De Gruyter Series in Nonlinear Analysis and Applications 3, Walter de Gruyter, Berlin, 1996.

\bibitem{SimonCompactness}
J.~Simon,
Compact sets in the space $L^p(0,T;B)$,
\emph{Ann. Mat. Pura Appl. (4)} \textbf{146} (1987), 65--96.

\bibitem{Solonnikov}
V.~A.~Solonnikov,
Estimates for solutions of non-stationary linearized systems of Navier--Stokes equations,
\emph{Amer. Math. Soc. Transl. Ser. 2} \textbf{75} (1968), 1--116.

\bibitem{Triebel}
H.~Triebel,
\emph{Theory of Function Spaces},
Monographs in Mathematics 78, Birkh\"auser, Basel, 1983.

\bibitem{XieFu}
H.~Xie and Y.~Fu,
Well-posedness in critical spaces for the density-dependent incompressible viscoelastic fluid system,
arXiv:1105.4690 [math.AP], 2011, doi:10.48550/arXiv.1105.4690.

\bibitem{YeLuoYin}
W.~Luo, W.~Ye and Z.~Yin,
The estimate of lifespan and local well-posedness for the non-resistive MHD equations in homogeneous Besov spaces,
\emph{Math. Methods Appl. Sci.} \textbf{48} (2025), no.~8, 8994--9006, doi:10.1002/mma.10770.

\bibitem{ZhangFangLp}
T.~Zhang and D.~Fang,
Global existence of strong solution for equations related to the incompressible viscoelastic fluids in the critical $L^p$ framework,
\emph{SIAM J. Math. Anal.} \textbf{44} (2012), no.~4, 2266--2288, doi:10.1137/110851742.

\bibitem{ZhangZhouHookean}
Z.~Zhang and Y.~Zhou,
Global well-posedness for incompressible Hookean elastodynamics in the critical Besov spaces,
\emph{Calc. Var. Partial Differential Equations} \textbf{64} (2025), no.~7, Paper No.~206, doi:10.1007/s00526-025-03075-6.

\end{thebibliography}
\end{document}